\documentclass{biometrika}

\usepackage{amsmath}

\usepackage{pgf,tikz} 
\usepackage{tikz-cd}
\usetikzlibrary{arrows.meta}
\usetikzlibrary{intersections}

\usepackage{graphicx}
\usepackage{float} 
\usepackage[hidelinks]{hyperref}

\usepackage{enumitem} 
\usepackage{nicematrix}

\usepackage{newtxtext}
\usepackage[subscriptcorrection]{newtxmath}

\usepackage{mathtools}

\DeclareRobustCommand{\unEdge}{%
  \mathrel{%
    \vcenter{%
      \hbox{%
        \tikz[baseline=0.0ex, x=1.5ex, y=1.5ex]{
          \draw[line width=0.15ex] (0,0) -- (1.2,0);
          \draw[line width=0.15ex] (0,0.25) -- (0,-0.25);
          \draw[line width=0.15ex] (1.2,0.25) -- (1.2,-0.25);
        }%
      }%
    }%
  }%
}
\tikzset{edgeblunt/.style = {{|[scale=0.5]}-{|[scale=0.5]}}}

\graphicspath{{./art/}}

\usepackage[plain,noend]{algorithm2e}

\makeatletter
\renewcommand{\algocf@captiontext}[2]{#1\algocf@typo. \AlCapFnt{}#2} 
\def\@algocf@capt@plain{top}
\renewcommand{\algocf@makecaption}[2]{%
  \addtolength{\hsize}{\algomargin}%
  \sbox\@tempboxa{\algocf@captiontext{#1}{#2}}%
  \ifdim\wd\@tempboxa >\hsize
    \hskip .5\algomargin%
    \parbox[t]{\hsize}{\algocf@captiontext{#1}{#2}}
  \else%
    \global\@minipagefalse%
    \hbox to\hsize{\box\@tempboxa}
  \fi%
  \addtolength{\hsize}{-\algomargin}%
}
\makeatother

\DeclareMathOperator*{\argmin}{arg\,min}

\begin{document}




\title{Identifiability and Estimation in Continuous Lyapunov Models}

\author{Cecilie Olesen Recke}
\affil{Department of Mathematical Sciences, University of Copenhagen,\\ 
\email{cor@math.ku.dk}}

\author{Niels Richard Hansen}
\affil{Department of Mathematical Sciences, University of Copenhagen,\\ 
\email{niels.r.hansen@math.ku.dk}}

\maketitle

\begin{abstract}
Cross-sectional observations from a dynamical system can be modeled via 
steady-state distributions of Markov processes. The major challenge is then to determine
whether the process parameters can be identified and estimated from the
steady-state distributions. We study this problem for continuous Lyapunov models
that arise as steady-state distributions of the solution to a multivariate stochastic 
differential equation, whose linear drift matrix is parametrized by a directed graph. 
We derive equations for the cumulant tensors of any order for this distribution, which
generalize the well-known covariance Lyapunov equation. Under a non-Gaussianity assumption
we prove generic identifiability of the drift matrix for any connected graph using the 
equations for the higher-order cumulants. Based on the identifiability result, we propose 
a new semiparametric estimator of the drift matrix, and we derive its asymptotic distribution.
A simulation study demonstrates the asymptotic validity of the estimator 
but shows that it is only accurate for relatively large sample sizes, illustrating the hardness of the unconstrained estimation problem.
\end{abstract}


\thispagestyle{empty}

\section{Introduction}

\subsection{Background}

The problem of estimating the parameters of a dynamical model from cross-sectional 
data has recently received considerable attention \citep{varando20a, Lorch:2024, guan2024identifyingdriftdiffusioncausal, bleile_efficient_2026}. An
important application is learning biological mechanisms from single cell data \citep{Wang:2023aa, Rohbeck:2024, lorch_latent_2026}, but this is challenging
with access to only temporal snapshots rather than dynamic trajectories. 
One largely outstanding question is to what extent the parameters are identifiable
from cross-sectional distributions. Partial results were given by 
\cite{Dettling:2023} for a Gaussian Markov process, and they showed, in particular, 
that the drift parameters of this process are only identifiable under strong assumptions. 

We consider a similar but more general class of Markov processes and 
give the first positive identifiability result regarding the drift parameters 
when the process is non-Gaussian. This relies on a thorough treatment of
the algebraic structures of higher-order cumulants determined by the model. 
We use our results to propose a new semiparametric estimator of 
the drift parameters and investigate its properties, asymptotically and
 in simulations.

\subsection{Main contributions}

We consider the $d$-dimensional steady-state distribution of the stationary Markov process that solves the stochastic differential equation (SDE)
\begin{equation}
    \label{eq::SDErep}
    dX_t = M X_t dt + dZ_t, 
\end{equation}
where $M$ is a stable $d \times d$ matrix and $Z = (Z_t)_{t \geq 0}$ is a Lévy process. Such a distribution is known as an $M$-selfdecomposable distribution \citep{Sato:1984, Masuda:2004}, 
and we regard it as a model of a single cross-sectional observation from the stationary process.
The entries of the matrix $M$ are the drift parameters of the process, and following \cite{varando20a}
we associate with any directed graph $G$ on nodes $\{1, \ldots, d\}$ 
a sparsity pattern of $M$: only entries corresponding to a directed edge
in $G$ can be non-zero, see Figure \ref{fig::ZeroPatternEx}. 
For a given graph $G$, our focus is on 
identification and estimation of $M$ from cross-sectional observations. 

We first derive, in Proposition \ref{prop::tensor-lyapunov}, the 
equations that determine all cumulants of the steady-state distribution from the 
matrix $M$ and the cumulants $\mathcal{C}_2, \mathcal{C}_3, \ldots$ of $Z_1$. Our main results are then Theorem \ref{thm::GenericIdResult} and Corollary \ref{cor::genericIDMoffdiag}. For those 
we effectively assume that the Lévy process has independent coordinates and is not 
a Brownian motion. If the graph $G$ is also connected then
$(M, \mathcal{C}_2, \mathcal{C}_r)$ is \emph{generically} identifiable up to a common scaling factor
from the covariance matrix and the $r$-th-order cumulant tensor. The possible exception set where identifiability fails is a proper algebraic subset of the parameter set. We refer to Theorem \ref{thm::GenericIdResult} and Corollary \ref{cor::genericIDMoffdiag} for the details, 
but note that identifiability up to a scaling factor is the best possible result without 
additional assumptions, such as assuming the cumulants of $Z_1$ known.

To prove Theorem \ref{thm::GenericIdResult} we derive trek rules of independent interest, collected in Appendix \ref{Appendix::treks}, that express how the $k$-th-order cumulants depend on $M$
and $\mathcal{C}_k$. This allows us to show that the equations from Proposition \ref{prop::tensor-lyapunov} can generically be solved uniquely for the parameters in terms of the cumulants. These equations naturally suggest an estimator of $M$, up to a scaling factor.
We show consistency and asymptotic normality of this estimator in Theorem \ref{thm::AsympoticsSingularValueEst}, with the most important contribution being the explicit formula for the asymptotic covariance matrix. The Julia package \url{https://github.com/nielsrhansen/SteadyStateStatistics.jl} includes an 
implementation. A simulation study shows that while 
the estimator behaves asymptotically according to the theory, it has a notable 
finite sample bias and the sample size needs to be fairly large for accurate 
estimation.

\subsection{Relations to existing literature}

It is well known that the covariance matrix $\Sigma$ for an $M$-selfdecomposable distribution
solves the continuous Lyapunov equation 
\begin{equation}
\label{eq::SecondOrderLyapunov}
    M \Sigma + \Sigma M^T + \mathcal{C}_2 = 0,
\end{equation}
where $\mathcal{C}_2$ is the covariance matrix of $Z_1$. A proof is given by
\cite{Jacobsen+1993+86+94} assuming the Lévy process is a Brownian motion, and similar
arguments apply in general once the SDE \eqref{eq::SDErep} is known to have a 
unique stationary solution. In fact, it 
also follows rather easily from Theorem 4.1 by \cite{Sato:1984} combined with their 
equations (2.15) and (2.18). Our Proposition~\ref{prop::tensor-lyapunov} generalizes 
this result and shows that the $k$-th-order cumulant solves
the \emph{$k$-th-order continuous Lyapunov equation}. Though this equation has been
studied previously \citep{Xu:2021}, its connection to the cumulants of $M$-selfdecomposable
distributions is new. 

Estimation of $M$ from cross-sectional observations, assuming a sparse graph and a diagonal 
$\mathcal{C}_2$, was
investigated by \cite{fitch2020learningdirectedgraphicalmodels}, \cite{varando20a} 
and \cite{pmlr-v236-dettling24a} using lasso-type estimators. They only considered 
the covariance equation \eqref{eq::SecondOrderLyapunov}, effectively assuming 
observations to be Gaussian.  
\cite{Dettling:2023} showed global identifiability from the second order equation \eqref{eq::SecondOrderLyapunov} when $\mathcal{C}_2$ is known and the 
graph is simple, but they also showed by example that there exist sparse graphs 
where $M$ is not even generically identifiable from \eqref{eq::SecondOrderLyapunov}. Still 
considering only the second order Lyapunov equation \eqref{eq::SecondOrderLyapunov},
\cite{vanseeventer2026signidentifiabilitycausaleffects} derived results on identification of 
just the signs of entries in $M$.

\cite{young2019identifying} and \cite{recke2026identifiabilitygraphicaldiscretelyapunov} derived related identifiability 
results for discrete Lyapunov models where the cumulant equations arise from the 
steady-state distribution of a vector autoregressive model. Only \cite{recke2026identifiabilitygraphicaldiscretelyapunov} considered the non-Gaussian setting. 

\cite{Lorch:2024} and \cite{bleile_efficient_2026} considered a setup similar to ours 
with cross-sectional observations from a stationary diffusion, that is, a 
stationary solution to an SDE driven by Brownian motion. They proposed 
and investigated different estimators of drift and diffusion parameters.
\cite{guan2024identifyingdriftdiffusioncausal} also investigated estimation
via cross-sectional observations from a diffusion, but they relied on
observations from multiple time points from a process out of its steady state 
to get identification. 
While these developments are closely related to the problem we consider, 
they are not directly applicable, because the solution to the SDE \eqref{eq::SDErep} is not a diffusion unless the Lévy process is a Brownian motion.

Our contributions are more inspired by the literature on 
identifiability of the mixing matrix in independent 
component analysis \citep{COMON1994287}, and its applications to linear non-Gaussian acyclic models \citep{JMLRv7shimizu06a}. Of particular relevance are the recent contributions by \cite{Zwiernik2024nonindependentcomponentsanalysis}
on identifiability and estimation of the mixing matrix using 
higher-order moments and cumulants. While their overall strategy is similar to ours, 
the algebraic techniques we use are quite different as the fundamental 
equations we study are different.

\section{The Continuous Lyapunov Model}

In this section we define the semiparametric \emph{continuous Lyapunov model}. Then we present the cumulant defining equations that we will use to show our main identifiability results for the parameters of interest, and which form the basis of our estimator.

The model consists of $d$-dimensional steady-state distributions of solutions to the SDE \eqref{eq::SDErep}, parametrized by the drift matrix $M$ and the distribution of the Lévy process $(Z_t)_{t \geq 0}$. There is a unique such steady-state distribution when $M$ is a stable matrix, 
that is, all its eigenvalues have negative real part, and when the Lévy process satisfies the 
weak integrability condition
\begin{equation}
    \label{eq:levy-int-cond}
    \mathbb{E}(\log(1 + \|Z_1\|)) < \infty.    
\end{equation}
This follows from Theorem 4.1 by \cite{Sato:1984}. The steady-state distribution is known as 
an \emph{$M$-selfdecomposable distribution}.

The drift matrix is the parameter of primary interest, and 
we will allow for it to have a particular sparsity structure according to a directed graph, see Figure \ref{fig::ZeroPatternEx}. Formally, with 
$[d] = \{1, \dots, d\}$ and with $G = ([d],E)$ a directed graph with node set $[d]$ 
and edge set $E$, we define the set
\begin{equation*}
    \mathbb{R}^{E} = \{ M \in \mathbb{R}^{d \times d} \; | \; M_{ij} = 0 \text{ if } j \rightarrow i \not \in E \}. 
\end{equation*}
We denote by $\mathbb{R}^E_{\mathrm{stab}}$ the subset of $\mathbb{R}^{E}$ where $M$
is also stable.

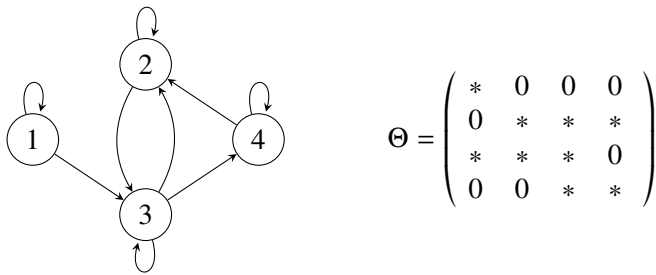
\begin{figure}
    \centering
\begin{center}
    \begin{tikzpicture}
      \node[circle, draw, minimum size=0.5cm] (1) at (-1.5,0) {1};
      \node[circle, draw, minimum size=0.5cm] (2) at (0,1) {2};
      \node[circle, draw, minimum size=0.5cm] (3) at (0,-1) {3};
      \node[circle, draw, minimum size=0.5cm] (4) at (1.5,0) {4};
      
      \draw[->,  >=stealth] (1) edge[loop above] (1);
      \draw[->,  >=stealth] (2) edge[loop above] (2);
      \draw[->,  >=stealth] (3) edge[loop below] (3);
      \draw[->,  >=stealth] (4) edge[loop above] (4);
      \draw[->,  >=stealth] (1) edge (3);
      \draw[->,  >=stealth] (4) edge (2);
      \draw[->,  >=stealth] (3) edge (4);
      \draw[->,  >=stealth] (2) edge [bend right] (3);
      \draw[->,  >=stealth] (3) edge[bend right] (2);

    \node at (5, 0) {$
    \Theta = \left( 
    \begin{array}{cccc}
     \;  * \; \; \; & 0 \; \; \; & 0 \; \; \; & 0 \; \;  \\
     \; 0 \; \; \; & * \; \; \; & * \; \; \; & * \; \;  \\
     \;  * \; \; \; & * \; \; \; & * \; \; \; & 0 \; \; \\
    \;  0 \; \; \; & 0 \; \; \; & * \; \; \; & * \; \; \\
    \end{array}
    \right) 
    $};
      
    \end{tikzpicture} 
    \end{center}
    \caption{Example of a directed graph $G = ([4], E)$ and corresponding zero-pattern of $M$ in $\mathbb{R}^{E}$}
    \label{fig::ZeroPatternEx}
\end{figure}

\begin{definition}
    \label{def::ContLyapunovLevy}
    Let $G = ([d],E)$ be a directed graph. The continuous Lyapunov model is the set of probability distributions $\mathcal{P}_G$ on $\mathbb{R}^d$ that are steady-state distributions of solutions to the SDE \eqref{eq::SDErep} for an $M \in \mathbb{R}^{E}_{\mathrm{stab}}$ and a Lévy process $(Z_t)_{t \geq 0}$ that satisfies the integrability condition \eqref{eq:levy-int-cond}.
\end{definition}

We will parametrize the distribution of the Lévy process by the distribution 
of $Z_1$. This infinitely divisible distribution, which we denote $Q_1$, 
uniquely determines the distribution of the entire process. We denote by 
$\mathcal{Q}$ the set of infinitely divisible distributions $Q_1$ on $\mathbb{R}^d$ satisfying 
the integrability condition
\begin{equation}
	\label{eq:levy-int-cond2}
	\int \log(1 + \|z\|)) Q_1(\mathrm{d}z) < \infty.
\end{equation}
For any graph $G$, the model $\mathcal{P}_G$ is 
therefore a set of $M$-selfdecomposable distributions parametrized by the finite 
dimensional parameter space $\mathbb{R}^{E}_{\mathrm{stab}}$ and the infinite dimensional 
space of infinitely divisible distributions $\mathcal{Q}$. We denote this parametrization by 
$$
    \psi : \mathbb{R}^{E}_{\mathrm{stab}} \times \mathcal{Q} \to \mathcal{P}_G.
$$

Provided that a distribution in $\mathcal{P}_G$ has finite $k$-th-order moment, we can
derive a simple generalization of the second-order Lyapunov equation \eqref{eq::SecondOrderLyapunov} that determines the $k$-th-order cumulant of the 
distribution. See \cite{mccullagh2018tensor} for background on cumulants.

In order to write these higher-order cumulant equations we use the notion of the $n$-mode product, also known as the Tucker product, between a tensor and a matrix. The $n$-mode product of a tensor $T \in \mathbb{R}^{I_1 \times I_2 \times \cdots \times I_N}$ with a matrix $M \in \mathbb{R}^{J \times I_n}$, denoted $T \times_n M$, 
is an $I_1 \times \cdots \times I_{n-1} \times J \times I_{n+1} \times \cdots \times I_N$ tensor with coordinate-wise entries 
\begin{equation*}
    (T \times_n M)_{i_1 \cdots i_{n-1} j i_{n+1} \cdots i_N} = \sum_{i_n = 1}^{I_n} T_{i_1 \cdots i_{n-1} i_n i_{n+1} \cdots i_{N}} M_{ji_n}.
\end{equation*}

\begin{proposition} \label{prop::tensor-lyapunov} Let $M$ be a $d \times d$ stable matrix
and let $Z = (Z_t)_{t \geq 0}$ denote a $d$-dimensional Lévy process with 
finite $k$-th-order moment. Then the corresponding $M$-selfdecomposable distribution 
has finite $k$-th-order moment, and the $k$-th-order cumulant tensor $\mathcal{K}$ solves the equation 
\begin{equation} \label{eq:k-lyapunov}
    \mathcal{K} \times_1 M + \ldots + \mathcal{K}  \times_k M + \mathcal{C}_k = 0
\end{equation}
where $\mathcal{C}_k = \mathrm{cum}_k(Z_1)$ is the $k$-th-order cumulant tensor of $Z_1$. Equation \eqref{eq:k-lyapunov} will be denoted the $k$-th-order continuous Lyapunov equation. 
\end{proposition}
The $k$-th-order continuous Lyapunov equation is linear in both the parameters $(M, \mathcal{C}_k)$ and
the $k$-th-order cumulant $\mathcal{K}$. Furthermore, given a stable $M$ and any $k$-th-order tensor $\mathcal{C}_{k}$, there exists a unique $k$-th-order tensor $\mathcal{K} $ solving \eqref{eq:k-lyapunov}. This follows from Corollary 3.1 by \cite{Xu:2021}. Appendix \ref{sec:lyapunov-appendix} includes
the proof of Proposition \ref{prop::tensor-lyapunov} and additional details
related to equation \eqref{eq:k-lyapunov}, such as the statement of Corollary 3.1 by \cite{Xu:2021} and an alternative proof of existence and uniqueness of the solution to equation \eqref{eq:k-lyapunov} that also applies to non-stable $M$.

For our identification results, we will focus on the combination of the covariance equation and one additional $r$-th-order continuous Lyapunov equation.
For estimation it is in practice most relevant to consider $r = 3$ or $r = 4$, or a combination to possibly achieve a more efficient estimator.

In order to show our identification results, we consider a subset of 
distributions in $\mathcal{P}_G$ where the $r$-th-order cumulants exist,
and where the second- and $r$-th-order cumulants of the Lévy process are diagonal with 
non-zero entries on the diagonal entries.

\begin{definition}
\label{dfn:P2rG}
Let $G = ([d], E)$ be a directed graph on $d$ nodes and let $r \geq 3$ be an integer. Denote by 
$\mathcal{P}^{2, r}_{G} \subseteq \mathcal{P}_G$ the subset of the continuous Lyapunov model 
for which the Lévy process has finite $r$-th-order moment and with cumulant tensors $\mathcal{C}_2$
and $\mathcal{C}_r$ being diagonal with non-zero diagonal entries.
\end{definition}

We let $\mathcal{Q}^{2, r}$ denote the subset of $\mathcal{Q}$ consisting of distributions with 
finite $r$-th-order moment and with the cumulant tensors $\mathcal{C}_2$ and $\mathcal{C}_r$ 
being diagonal with non-zero diagonal entries. Then  
$$
    \mathcal{P}^{2, r}_{G} = \psi(\mathbb{R}^{E}_{\mathrm{stab}} \times \mathcal{Q}^{2, r}),
$$
that is, the submodel $\mathcal{P}^{2, r}_{G}$ is parametrized by 
$\mathbb{R}^{E}_{\mathrm{stab}} \times \mathcal{Q}^{2, r}$ via the parametrization $\psi$.
If the Lévy process has \emph{independent} coordinates, that is, if $Q_1$ is a product measure, 
any cumulant tensor is diagonal, provided it exists. Since we only need that $\mathcal{C}_2$ and $\mathcal{C}_r$ are diagonal, we avoid making the stronger independence assumption in the 
definition of $\mathcal{P}^{2, r}_{G}$. 

We also assume that the diagonal entries of 
$\mathcal{C}_2$ and $\mathcal{C}_r$ are non-zero. For $\mathcal{C}_2$ this assumption is a 
non-degeneracy condition, while for $\mathcal{C}_r$ this is the specific non-Gaussianity assumption 
we need to establish identifiability. 

In addition to the abstract parameter and model spaces we will also introduce the
finite-dimensional spaces of cumulants resulting from the model spaces. Let 
$\mathrm{Sym}^k(\mathbb{R}^d)$ denote the set of $k$-th-order 
$d \times d \times \ldots \times d$ symmetric tensors, and let $\mathrm{Sym}^k(\mathbb{R}^d)^*$ 
denote the subset of tensors whose diagonal entries are non-zero. For 
$k = 2$ we let $\mathrm{PD}_d \subseteq \mathrm{Sym}^2(\mathbb{R}^d)^*$ 
denote the subset of symmetric $d \times d$-matrices that are also positive definite.

\begin{definition}
\label{dfn:cumulantmodel}
Let $G = ([d], E)$ be a directed graph with $d$ nodes and let $r \geq 3$ be an integer. 
The second- and $r$-th-order Lyapunov cumulant model is the set of second- and $r$-th-order tensors in the set 
{\small
\begin{equation*}
\mathcal{M}^{2, r}_G = 
 \left\{ 
(\Sigma, \mathcal{K} ) \in (\mathrm{PD}_p, \mathrm{Sym}^r(\mathbb{R}^d)) \; \left\vert \; \begin{aligned}
    &M \Sigma + \Sigma M^T + \mathcal{C}_2 = 0, \\ 
    &\mathcal{K}  \times_1 M +  \cdots+ \mathcal{K}  \times_r M + \mathcal{C}_r = 0, \\ 
    &M \in \mathbb{R}^{E}_{\mathrm{stab}}, \text{ and } \mathcal{C}_2 \in \mathrm{PD}_{d}, 
    \mathcal{C}_r \in \mathrm{Sym}^r(\mathbb{R}^d)^* \text{ diagonal}
\end{aligned}\right.
\right\}.
\end{equation*}
}
\end{definition}

Let $\mathrm{cum}_k(P) \in \mathrm{Sym}^k(\mathbb{R}^d)$ denote the $k$-th-order cumulant for a 
probability distribution $P$ with finite $k$-th-order moment. Then for 
$(M, Q_1) \in \mathbb{R}^E_{\mathrm{stab}} \times \mathcal{Q}$ with $Q_1$ having finite $k$-th-order 
moment, 
$$
    \mathrm{cum}_k(\psi(M, Q_1)) = \mathcal{K} 
$$
with $\mathcal{K}$ solving \eqref{eq:k-lyapunov} for $\mathcal{C}_k = \mathrm{cum}_k(Q_1).$ 
Letting additionally $\mathrm{cum}_{2,r}(P) = (\mathrm{cum}_2(P), \mathrm{cum}_r(P))$, we see that 
$$
    \mathrm{cum}_{2,r} : \mathcal{P}^{2, r}_G \to \mathcal{M}^{2, r}_G 
$$
maps the distribution $P$ in the model $\mathcal{P}^{2, r}_G$ to its corresponding 
pair $(\Sigma, \mathcal{K} )$ of second- and $r$-th-order cumulants in $\mathcal{M}^{2, r}_G$. See the commutative diagram in Figure \ref{fig:comdia} for an overview of these maps and sets. 
While $\mathrm{cum}_{2,r}$ may not be surjective for all $r$, we show that 
it is for $r$ odd, and for $r$ even its image has full dimension, see Proposition \ref{prop::piSurjective} in Appendix \ref{sec:AuxResultsCompoundPoisson} in combination with the commutative diagram in Figure \ref{fig:comdia}.

\section{Identifiability}

\subsection{Overview}

Usually, the parameter $M$ would be said to be identified from the probability 
measure $P \in \mathcal{P}_G$ if it were uniquely determined by $P$. Section \ref{sec:lim} shows
that this is not possible and that we can only hope to identify 
$M$ up to scaling in general. We give necessary conditions, including that $G$ is connected, 
for identification up to a global scaling factor.

Our main Theorem~\ref{thm::GenericIdResult} shows that the conditions 
are generically essentially sufficient as well and describes the precise
injectivity properties of the parameterization of the cumulants via the Lyapunov equations. 
The conclusion is rephrased in Corollary~\ref{cor::genericIDMoffdiag} as
$(M, \mathcal{C}_2, \mathcal{C}_r)$ being generically identifiably, up to a joint 
scaling of these parameters, from a combination of the covariance matrix 
and one $r$-th-order cumulant tensor.

The proof of Theorem~\ref{thm::GenericIdResult} exploits the linearity of the continuous 
Lyapunov equations, which allows us to transfer 
the question of identifiability to rank questions for certain matrices. We describe the 
necessary reorganization of the Lyapunov equations in full detail in Section \ref{sec:proof-reorganize}. Section \ref{sec::Estimation} on estimation also relies heavily on the reorganization. 
The non-trivial, but technical, arguments about rank properties are given in 
Appendix \ref{Appendix::MainresultLemmas}. When $G$ is not connected, Section \ref{sec:nonid} shows 
that identification up to a global scaling factor fails, and that identification is only 
possible up to a scaling factor for each connectivity component.

\subsection{Limitations to identifiability}
\label{sec:lim}

For the parametrized model $\mathcal{P}_G$, the standard notion of identifiability is simply the 
question of whether the map $\psi : \mathbb{R}^{E}_{\mathrm{stab}} \times \mathcal{Q} \to \mathcal{P}_G$ is injective. To see that $\psi$ is not injective, we write \eqref{eq::SDErep} as a stochastic 
integral equation and do a change-of-variable in the integration to get 
$$
    X_{ct} = X_0 + \int_0^t (cM) X_{cs} \mathrm{d} s + Z_{ct}
$$
for any $c > 0$. This shows that $\psi(M, Q_1) = \psi(cM, Q_c)$ for all $c > 0$, where $Q_c$ is the distribution of $Z_c$, and $\psi$ is thus not injective. The best possible 
identification result would therefore be that 
$\psi^{-1}(\psi(M, Q_1)) = \{ (cM, Q_c) \mid c > 0 \},$
that is, that the parameter $(M, Q_1)$ is identifiable up to a common scaling factor. 

From the Lévy-Khintchine representation of $Q_c$ it likewise follows that 
$\mathrm{cum}_k(Z_{c}) = c \mathcal{C}_k$ for any $k$. We 
see that the lack of identifiability is also reflected in the cumulant equations 
\eqref{eq:k-lyapunov}: if $(M, \mathcal{C}_k)$ is a solution with $M$ stable for given $\mathcal{K}$ then
$(cM, c\mathcal{C}_k)$ is a solution with $cM$ stable for all $c > 0$. 
This particular lack of identifiability is intuitively reasonable, 
since we cannot identify the absolute speed of the stationary solution of 
\eqref{eq::SDErep} from observations at a single time point.

Since $M$ is our parameter of interest, we are primarily interested in whether there exists 
a map
\begin{equation}
    \label{eq:idmap}
    P = \psi(M, Q_1) \mapsto \{c M \mid c > 0\}
\end{equation}
from $\mathcal{P}_G$ to the equivalence classes in $\mathbb{R}^{E}_{\mathrm{stab}}$ of 
matrices identical up to scaling. Such a map cannot exists in general either, 
because if $G$ has at least two connectivity components, we can 
rescale the parameters for each component independently, see also Corollary \ref{cor::RankConnectedComponents}. Additionally, even if $G$ is connected, but 
$Q_1 = \mathcal{N}(0, \mathcal{C}_2)$ is a Gaussian distribution with diagonal 
covariance matrix $\mathcal{C}_2,$ then $M$ cannot always be identified up to 
a scaling from $P = \psi(M, Q_1)$. Indeed, \cite{Dettling:2023} showed this by their 
Example 8.7 for the graph in Figure \ref{fig::UnIDCov} for any choice of diagonal 
$\mathcal{C}_2$.

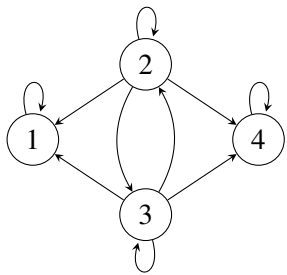
\begin{figure}
    \centering
\begin{center}
    \begin{tikzpicture}
      \node[circle, draw, minimum size=0.5cm] (1) at (-1.5,0) {1};
      \node[circle, draw, minimum size=0.5cm] (2) at (0,1) {2};
      \node[circle, draw, minimum size=0.5cm] (3) at (0,-1) {3};
      \node[circle, draw, minimum size=0.5cm] (4) at (1.5,0) {4};
      
      \draw[->,  >=stealth] (1) edge[loop above] (1);
      \draw[->,  >=stealth] (2) edge[loop above] (2);
      \draw[->,  >=stealth] (3) edge[loop below] (3);
      \draw[->,  >=stealth] (4) edge[loop above] (4);
      \draw[->,  >=stealth] (2) edge (1);
      \draw[->,  >=stealth] (3) edge (1);
      \draw[->,  >=stealth] (2) edge (4);
      \draw[->,  >=stealth] (3) edge (4);
      \draw[->,  >=stealth] (2) edge [bend right] (3);
      \draw[->,  >=stealth] (3) edge[bend right] (2);
    \end{tikzpicture} 
    \end{center}
    \caption{Directed graph, $G$, such that the drift matrix, $M$, is \emph{not} identifiable 
    up to scaling from the Gaussian model $\psi(M, \mathcal{N}(0, \mathcal{C}_2))$ for any
    diagonal $\mathcal{C}_2$ \citep{Dettling:2023}.}
    \label{fig::UnIDCov}
\end{figure}

Our goal is to find a submodel of $\mathcal{P}_G$ such that we can establish the existence 
of the map \eqref{eq:idmap} defined on this submodel, provided that $G$ is connected. 
A candidate for such a submodel is $\mathcal{P}^{2, r}_{G}$ from Definition~\ref{dfn:P2rG}, 
which rules out all Gaussian distributions. It turns out that we need to further remove 
an exception set from $\mathcal{P}^{2, r}_{G}$, but we can argue that this set is, 
in a certain sense, negligible. The precise way to phrase this is via the 
finite-dimensional parametrization of the second- and $r$-th-order Lyapunov cumulant model 
$\mathcal{M}^{2,r}_G$ in terms of $(M, \mathcal{C}_2, \mathcal{C}_r)$ of which the negligible set will be an algebraic subset. 

\subsection{Main results}

To simplify notation, we identify \emph{diagonal} $\mathcal{C}_2$ and $\mathcal{C}_r$ tensors with 
their diagonals as elements in $\mathbb{R}^{d}_{+}$ and $(\mathbb{R} \setminus \{ 0 \})^d$,
respectively, and we introduce the parameter set 
\begin{equation*}
    \Theta_G = \mathbb{R}^{E}_{\mathrm{stab}} \times \mathbb{R}^{d}_{+} \times (\mathbb{R} \setminus \{ 0 \})^d. 
\end{equation*}
The set $\Theta_{G}$ is an open set in standard Euclidean topology. 
Elements in $\Theta_{G}$ are denoted \mbox{$\theta = (M, \mathcal{C}_2, \mathcal{C}_r)$.}
The question of identifiability is then settled in terms of the rational parametrization  
\begin{align*}
    \varphi_{G,r}: \Theta_G &\rightarrow \mathcal{M}^{2,r}_G \\
    (M, \mathcal{C}_2, \mathcal{C}_r) &\mapsto (\Sigma, \mathcal{K} ),
\end{align*}
where $\Sigma$ and $\mathcal{K} $ are the unique solutions to the corresponding second- and $r$-th-order Lyapunov equation, as guaranteed to exist by Proposition \ref{prop::XuResult}. By Definition \ref{dfn:cumulantmodel}, $\varphi_{G,r}$ is surjective. 

We can now state our main 
identifiability result pertaining to the injectivity properties of $\varphi_{G,r}$.

\begin{theorem}
    \label{thm::GenericIdResult}
    Let $d \geq 2$,  $r \geq 3$ be integers and let $G = ([d], E)$ be any connected graph with all self-loops present. 
    Then there exists a proper algebraic subset $\mathcal{N}_G \subset \Theta_G$ such that for $\theta \in \Theta_G \setminus \mathcal{N}_G$,
    \begin{equation}
        \label{eq::PrecisestatementaboutDimofFiber}
        \varphi^{-1}_{G,r}(\varphi_{G,r}(\theta)) = \{ c \theta \; | \; c > 0 \}. 
    \end{equation}
\end{theorem}

There are different notions in the literature of weak forms of identifiability.
See Chapter 16 by \cite{sullivant2018algebraic} for an overview.
A result such as Theorem \ref{thm::GenericIdResult} can be    
phrased as $\theta$ being \emph{generically} identifiable up to scaling. The 
term \emph{generic} means that the possible exception set, where $\theta$
cannot be identified, is negligible in some sense. In Theorem 
\ref{thm::GenericIdResult} the negligible set $\mathcal{N}_G$ is a proper 
algebraic subset, which means that it is a proper subset and the 
vanishing set of a collection of polynomials, i.e., a variety. This makes $\mathcal{N}_G$ a Lebesgue 
null set in the ambient Euclidean space, but more importantly,
since $\Theta_G$ is open, $\mathcal{N}_G$ is a strictly lower-dimensional 
set, it is closed and $\Theta_G \backslash \mathcal{N}_G$ is open and 
dense in $\Theta_G$. For these reasons, \eqref{eq::PrecisestatementaboutDimofFiber}
is regarded as holding for most points in $\Theta_G$, and it is reasonable 
to regard the exception set 
$\mathcal{N}_G$ as negligible. Using the conclusion of 
Theorem \ref{thm::GenericIdResult} as the definition of what we mean by generic 
identifiability, we get the following corollary.

\begin{corollary}
    \label{cor::genericIDMoffdiag}
    Let $d \geq 2$, $r \geq 3$ be integers and let $G = ([d], E)$ be any connected graph with all self-loops present. Then the drift matrix $M$ and the diagonal and non-zero cumulants $\mathcal{C}_2$ and $\mathcal{C}_r$ are jointly generically identifiable from 
    the cumulant tensors $(\Sigma, \mathcal{K} ) \in \mathcal{M}^{2,r}_G$ up to a common scaling. 
\end{corollary}

\begin{remark} \label{rem::Selfloops}
    Theorem \ref{thm::GenericIdResult} and Corollary \ref{cor::genericIDMoffdiag} make two assumptions about the graph; that it is connected and that it contains all self-loops. Section \ref{sec:nonid}
    argues that the connectedness assumption is necessary. If the graph is acyclic, the drift matrix 
    can also only be stable if it contains all self-loops, and our proof relies on the existence of 
    an acyclic subgraph corresponding to a stable $M$. The assumption that the graph contains 
    all self-loops is harmless from a practical perspective as it is almost always 
    sensible to assume that a variable affects itself. However, it is of theoretical 
    interest to determine if the assumption can be relaxed. We have by computation found 
    examples of cyclic graphs for which Theorem \ref{thm::GenericIdResult} remains true for 
    $r = 3$ even though some self-loops are removed, see Figure \ref{fig::Idwithmissingwselfloops} for two examples on three nodes. 
\end{remark} 


Theorem \ref{thm::GenericIdResult} and Corollary \ref{cor::genericIDMoffdiag} do 
not explicitly answer to what extent the map \eqref{eq:idmap} exists. To do so 
we introduce the map 
\begin{align*}
    \pi_{2,r} : & \mathbb{R}^E_{\mathrm{stab}} \times \mathcal{Q}^{2, r} \to 
    \Theta_G \\
    & (M, Q_1) \mapsto (M, \mathrm{cum}_2(Q_1), \mathrm{cum}_r(Q_1)),
\end{align*}
which makes the diagram in Figure \ref{fig:comdia} commute. Under the 
assumptions of Theorem \ref{thm::GenericIdResult}, the map \eqref{eq:idmap} can then be defined on $\psi(\mathbb{R}^E_{\mathrm{stab}} \times \mathcal{Q}^{2, r} \backslash \pi_{2,r}^{-1}(\mathcal{N}_G))$ by composing $\varphi_{G, r}^{-1} \circ \mathrm{cum}_{2,r}$ with a projection onto the first coordinate. The extent to 
which $\pi_{2,r}^{-1}(\mathcal{N}_G)$ is negligible in 
$\mathbb{R}^E_{\mathrm{stab}} \times \mathcal{Q}^{2, r}$ is not entirely clear. 
We do, however, show Proposition \ref{prop::piSurjective} in 
Appendix \ref{sec:AuxResultsCompoundPoisson} stating that $\pi_{2,r}$ is 
surjective if $r$ is odd, and the image of $\pi_{2,r}$ is, at least, of 
full dimension if $r$ is even. Importantly, the image of $\pi_{2,r}$ in 
$\Theta_G$ is therefore not a lower dimensional set like $\mathcal{N}_G$, 
which justifies that generic identifiability in the sense of 
Theorem \ref{thm::GenericIdResult} and Corollary \ref{cor::genericIDMoffdiag}
is of relevance.

\begin{figure}
    \centering
\begin{tabular}{cp{1cm}c}
    \begin{tikzpicture}
      \node[circle, draw, minimum size=0.5cm] (1) at (1,1) {1};
      \node[circle, draw, minimum size=0.5cm] (2) at (0,0) {2};
      \node[circle, draw, minimum size=0.5cm] (3) at (2,0) {3};

      \draw[->,  >=stealth] (3) edge[loop below] (3);
      \draw[->,  >=stealth] (2) edge[loop below] (2);
      \draw[->,  >=stealth] (1) edge [bend left] (2);
      \draw[->,  >=stealth] (2) edge[bend left]  (1);
      \draw[->,  >=stealth] (2) edge[bend right]  (3);
    \end{tikzpicture} & &
   \begin{tikzpicture}
      \node[circle, draw, minimum size=0.5cm] (1) at (1,1) {1};
      \node[circle, draw, minimum size=0.5cm] (2) at (0,0) {2};
      \node[circle, draw, minimum size=0.5cm] (3) at (2,0) {3};

      \draw[->,  >=stealth] (3) edge[loop below] (3);
      \draw[->,  >=stealth] (1) edge [bend left] (2);
      \draw[->,  >=stealth] (2) edge[bend left]  (1);
      \draw[->,  >=stealth] (2) edge[bend right]  (3);
      \draw[->,  >=stealth] (3) edge[bend right]  (1);
    \end{tikzpicture}
\end{tabular}
    \caption{Two directed graphs on three nodes with one and two self-loops missing, respectively, for which $M$ is still identifiable in the sense of Theorem \ref{thm::GenericIdResult} for $r = 3$.}
    \label{fig::Idwithmissingwselfloops}
\end{figure}
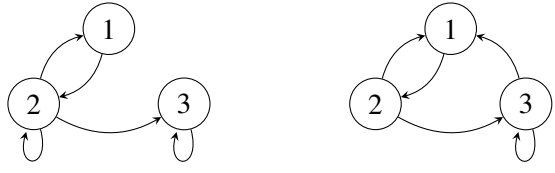

\begin{figure}
\centering 
\begin{tabular}{cp{1cm}c}
\begin{tikzcd}[column sep=large,row sep=large]
\mathbb{R}^E_{\mathrm{stab}} \times \mathcal{Q}^{2, r} \arrow[r, "\psi"] \arrow[d, "\pi_{2,r}"'] & 
\mathcal{P}^{2, r}_{G} \arrow[d, "\mathrm{cum}_{2,r}"] \\
\Theta_G \arrow[r,"\varphi_{G,r}"] & \mathcal{M}^{2,r}_G 
\end{tikzcd} & &
\begin{tikzcd}[column sep=large,row sep=large]
(M, Q_1) \arrow[r, "\psi"] \arrow[d, "\pi_{2,r}"'] & 
P \arrow[d, "\mathrm{cum}_{2,r}"] \\
(M, \mathcal{C}_2, \mathcal{C}_r) \arrow[r,"\varphi_{G,r}"] & 
(\Sigma, \mathcal{K})
\end{tikzcd}
\end{tabular}
\caption{Relations between model spaces and parametrizations as a commutative diagram.}
\label{fig:comdia}
\end{figure}

\subsection{Outline of the proof of Theorem \ref{thm::GenericIdResult} via reorganization of the Lyapunov equations}
\label{sec:proof-reorganize}

In this section we reorganize the second- and $r$-th-order cumulant equations from 
Proposition~\ref{prop::tensor-lyapunov}
so that they appear as a homogeneous system of linear equations in the vectorization of the parameters 
$M$, $\mathcal{C}_2$ and $\mathcal{C}_r$. For the vectorization we do not assume that 
$\mathcal{C}_2$ and $\mathcal{C}_r$ are diagonal. Since the (non-zero) parameter vector is in the kernel of the 
matrix that defines the linear equations, the conclusion of Theorem \ref{thm::GenericIdResult}, equation \eqref{eq::PrecisestatementaboutDimofFiber},
is true whenever this kernel is one-dimensional. This reduces our question about identifiability 
to a question about the rank of a matrix. To illustrate the idea behind the general proof,
Example \ref{ex:twonodes} gives the proof in the case of the complete graph on two nodes and with $r = 3$. 


In order to vectorize the cumulant equations we use a so called symmetric, or unique, vectorization 
operator. That is, it is a vectorization of only the unique entries of the symmetric covariance matrix and the symmetric $r$-th-order cumulant tensor. We denote the unique vectorization operator by $\mathrm{vec}_u$ and the ordinary vectorization operator by $\mathrm{vec}$.

Since the $k$-th order Lyapunov equation is linear in $M$, its unique vectorization can be written as 
\begin{equation}
    \label{eq::VecThirdOrder}
    A_k(\mathcal{K} ) \mathrm{vec}(M) + \mathrm{vec}_u(\mathcal{C}_k) = 0
\end{equation}
for some matrix $A_k(\mathcal{K} )$ depending on the $k$-th-order cumulant $\mathcal{K} $. 
To understand how $A_k(\mathcal{K})$ is organized, we first consider 
the case $k = 2$, where the equation is 
\begin{equation}
    \label{eq::VecSecondOrder}
    A_2(\Sigma) \mathrm{vec}(M) + \mathrm{vec}_u(\mathcal{C}_2) = 0
\end{equation}
with $A_2(\Sigma)$ a $(d(d+1)/2) \times d^2$ matrix. To write down the entries of 
this matrix, we index the second-order Lyapunov equation and 
the unique vectorization of $\mathcal{C}_2$ by $\{(i_1i _2) \; | \; i_1 \leq i_2 \}$. 
We index the columns of $A_2$ by the $d^2$ \emph{potential} edges $\alpha \rightarrow \beta$ of
a graph, corresponding to all possible entries of $M$. 
The expressions for the entries are then 
\begin{align}
\label{eq::A2SigmaEntries}
    A_2(\Sigma)_{(i_1 i_2), (\alpha \rightarrow \beta)} = \begin{cases}
        0 & \mathrm{ if } \; \beta  \neq i_1, i_2\\
        2 \Sigma_{i_1 i_2} & \mathrm{ if } \; \beta = i_1 = i_2\\
        \Sigma_{i_2 \alpha} & \mathrm{ if } \; \beta = i_1 \neq i_2\\
        \Sigma_{i_1 \alpha } & \mathrm{ if } \; \beta = i_2 \neq i_1, 
    \end{cases}
\end{align}
see also equation (4.3) by \cite{Dettling:2023}.

Similarly, the $k$-th-order Lyapunov equation and unique vectorization of $\mathcal{C}_k$ are indexed by $\{(i_1 \dots i_k) \; | \; i_1 \leq \cdots \leq i_k \}$. The matrix $A_k(\mathcal{K} )$ has dimensions 
$$
    \binom{d+(k-1)}{k} \times d^2.
$$
In general, the entry  
$A_k(\mathcal{K} )_{(\beta i_2 \dots i_k), (\alpha \rightarrow \beta)}$ is an integer multiple of $\mathcal{K} _{\alpha i_2 \dots i_k}$, and likewise for $\beta$ equal to one of the other indices. The remaining entries 
are zero. 
To determine the integer, one needs to count how many of the mode products in the equation 
will contribute with the same term, which will depend on how many of the $i$-indices are equal to $\beta$. Thus, again written in the case $\beta = i_1$, the general expression is given by
\begin{equation} \label{eq::generalAmatrixentry}
    A_k(\mathcal{K} )_{(i_1 \dots i_k), (\alpha \rightarrow \beta)} = N((i_1 \dots i_k), \beta) \cdot \mathcal{K} _{\alpha i_2 \dots i_k},
\end{equation}
where $N((i_1 \dots i_k), \beta)$ is the number of row indices $i_1 \leq \ldots \leq i_k$
equal to $\beta$.

Writing this out, we obtain a complete description of $A_3(\mathcal{K} )$ by
\begin{equation}
\label{eq::GeneralA3}
    A_3(\mathcal{K} )_{(i_1i_2i_3), (\alpha \rightarrow \beta)} = 
    \begin{cases}
         0 & \qquad  \beta \neq i_1, i_2, i_3 \\ 
         3 \mathcal{K} _{i_1 i_2 i_3}  & \qquad \beta = i_1 = i_2 = i_3 \\ 
         2\mathcal{K}_{i_1 \alpha i_3} & \qquad  \beta = i_1 = i_2 \neq i_3 \\ 
         2\mathcal{K} _{ \alpha i_2 i_3} &  \qquad \beta = i_2 = i_3 \neq i_1 \\ 
         \mathcal{K} _{i_1 i_2 \alpha } & \qquad \beta = i_3, i_2 \neq i_3\\
         \mathcal{K} _{\alpha i_2 i_3} & \qquad \beta = i_1, i_1 \neq i_2\\
         \mathcal{K} _{i_1 \alpha i_3} & \qquad \beta = i_2, i_2 \neq i_3, i_2 \neq i_1. \\ 
    \end{cases}
\end{equation}

    Again, letting $r \geq 3$ be a fixed integer, we can combine the second-order vectorized equation, \eqref{eq::VecSecondOrder} and \eqref{eq::VecThirdOrder} for $k = r$ to obtain the following linear system, still allowing $\mathcal{C}_2$ and $\mathcal{C}_r$ to be non-diagonal,
    \begin{align}
    \label{eq::BigLinSys}
        \begin{pmatrix}
            A_2(\Sigma) & I_{d(d+1)/2} & 0 \\ 
            A_r(\mathcal{K} ) & 0 & I_{\binom{d+(r-1)}{r}} 
        \end{pmatrix} 
        \begin{pmatrix}
            \mathrm{vec}(M) \\ 
            \mathrm{vec}_u(\mathcal{C}_2)\\
            \mathrm{vec}_u(\mathcal{C}_r)
        \end{pmatrix} = 0.
    \end{align}
   The question of identifiability is now reduced to considering the rank of the matrix above.  
   Without any restrictions on the cumulants $\mathcal{C}_2$ and $\mathcal{C}_r$ there are 
   as many unique entries in these as there are unique equations. Therefore, it is impossible to 
   establish identifiability if we allow $M$, $\mathcal{C}_2$ and $\mathcal{C}_r$ to be arbitrary. 
   This would not change by adding additional equations for other higher-order cumulants, since we would always be adding the same number of equations as unknown parameters. Therefore, certain constraints
   are necessary on the entries of $M$, $\mathcal{C}_2$ or $\mathcal{C}_r$. 
   
    For $\mathcal{C}_2$ and $\mathcal{C}_r$ diagonal, \eqref{eq::BigLinSys} reduces to 
    \begin{equation}
    \label{eq::ThmEquationSys0}
        \begin{pmatrix}
            A_2(\Sigma)_{\mathrm{off}} & 0 & 0 \\ 
            A_r(\mathcal{K} )_{\mathrm{off}} & 0 & 0 \\
            A_2(\Sigma)_{\mathrm{diag}} & I_{d} & 0 \\
            A_r(\mathcal{K} )_{\mathrm{diag}} & 0 & I_{d} \\ 
        \end{pmatrix} 
        \begin{pmatrix}
            \mathrm{vec}(M) \\
            \mathrm{diag}(\mathcal{C}_2 )\\
             \mathrm{diag}(\mathcal{C}_r)\\
        \end{pmatrix}
         = 0,
    \end{equation}
    where $A_2(\Sigma)_{\mathrm{off}}$ and $A_r(\mathcal{K} )_{\mathrm{off}}$ denote all the rows indexed by the off-diagonal tensor entries of $A_2(\Sigma)$ and $A_r(\mathcal{K} )$, respectively. That is, all the rows not indexed by $(ii)$ or $(i \dots i)$. Similarly, $A_2(\Sigma)_{\mathrm{diag}}$ and $A_r(\mathcal{K} )_{\mathrm{diag}}$ denote the rows indexed by the diagonal entries. In the proof of Theorem~\ref{thm::GenericIdResult}
    we establish that the upper left block of the matrix in \eqref{eq::ThmEquationSys0},
     \begin{equation}
        \label{eq:nd}
        A_{\mathrm{off}} = \begin{pmatrix}
            A_2(\Sigma)_{\mathrm{off}}\\ 
            A_r(\mathcal{K} )_{\mathrm{off}} 
        \end{pmatrix},
    \end{equation}
    generically has rank $d^2 - 1$. This shows the generic identifiability of $M$ up to scaling, 
    and from the remaining equations in \eqref{eq::ThmEquationSys0}, the diagonal cumulants $\mathcal{C}_2$
    and $\mathcal{C}_r$ are also identified up to scaling. 

    To prove that $A_{\mathrm{off}}$ generically has rank $d^2-1$ it is enough to exhibit one choice of \mbox{$\theta = (M, \mathcal{C}_2, \mathcal{C}_3) \in \Theta_G$} such that $A_{\mathrm{off}}$ has rank $d^2-1$. This is because the entries of the cumulants, $\Sigma$ and $\mathcal{K}$, can be written as rational functions in the parameters $M$, $\mathcal{C}_2$ and $\mathcal{C}_r$, since the continuous Lyapunov equations are also linear in the respective cumulants, see Proposition \ref{prop::VecOfEquation} in Appendix \ref{sec:lyapunov-appendix}. Alternatively, this can also be seen using the general trek rule in Appendix \ref{Appendix::treks}. Thus, any determinant of a submatrix of $A$ will also be a rational function in $M$, $\mathcal{C}_2$ and $\mathcal{C}_r$. Therefore, the determinant of a submatrix of $A$ being non-zero corresponds to a polynomial in $M$, $\mathcal{C}_2$ and $\mathcal{C}_r$ being non-zero, and to show that a polynomial is not the zero polynomial it is enough to show that it is non-zero in a single point.

\begin{example} 
    \label{ex:twonodes}
    To illustrate by example how we prove Theorem \ref{thm::GenericIdResult}, 
    we consider the smallest non-simple graph, the complete graph on $d = 2$ nodes, see Figure \ref{fig::TwoCycle}, and 
    prove that \eqref{eq:nd} for $r = 3$ generically has rank $d^2 -1 = 3$. For 
    diagonal $\mathcal{C}_2$ and $\mathcal{C}_3$, the equation system in equation 
    \eqref{eq::BigLinSys} can be written as
\begin{equation*}
    \begin{pNiceArray}{cccc|cccc}[first-row,first-col, nullify-dots]
       & 1 \rightarrow 1    & 1 \rightarrow 2 & 2 \rightarrow 1       & 2 \rightarrow 2 & (\mathcal{C}_2)_{11} & (\mathcal{C}_2)_{22} &  (\mathcal{C}_3)_{111} & (\mathcal{C}_3)_{222}  \\
(11)    & 2 \Sigma_{11} & 0 & 2 \Sigma_{12}  & 0 & 1 & 0 & 0 & 0   \\
(12)  & \Sigma_{12} & \Sigma_{11} & \Sigma_{22} & \Sigma_{12} & 0  & 0  & 0 & 0  \\
(22)     & 0 &  2 \Sigma_{12} & 0 & 2 \Sigma_{22} & 0  & 1 & 0  & 0       \\
(111)      & 3 \mathcal{K} _{111} & 0  & 3 \mathcal{K} _{112} & 0 & 0  & 0  & 1 & 0       \\
(112)      & 2 \mathcal{K} _{112} & \mathcal{K} _{111} & 2 \mathcal{K} _{122} & \mathcal{K} _{112} & 0 & 0 & 0  & 0       \\
(122)      & \mathcal{K} _{122} & 2 \mathcal{K} _{112} & \mathcal{K} _{222} & 2\mathcal{K} _{122}  & 0 & 0 & 0 & 0       \\
(222)      & 0 & 3 \mathcal{K} _{122} & 0 & 3\mathcal{K} _{222} & 0 & 0  & 0  & 1       \\
\end{pNiceArray}
 \begin{pmatrix}
    M_{11} \\
    M_{21} \\ 
    M_{12} \\ 
    M_{22} \\ 
    (\mathcal{C}_2)_{11} \\
    (\mathcal{C}_2)_{22} \\
    (\mathcal{C}_3)_{111} \\
    (\mathcal{C}_3)_{222} \\
\end{pmatrix} = 0. 
\end{equation*}
Because of the sparsity structure exhibited by the $\mathcal{C}$ columns, the matrix has full rank minus one if the submatrix with row indices corresponding to three off-diagonal entries of the tensors has rank three. These are rows two, five and six. Removing the four $\mathcal{C}$ columns as well, we 
arrive at the $3 \times 4$ matrix 

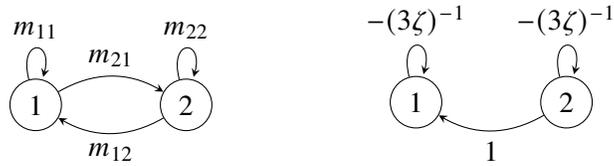
\begin{figure}
    \centering
\begin{tabular}{cp{1cm}c}
    \begin{tikzpicture}
      \node[circle, draw, minimum size=0.5cm] (1) at (0,0) {1};
      \node[circle, draw, minimum size=0.5cm] (2) at (2,0) {2};

      \draw[->,  >=stealth] (1) edge[loop above] node[midway, above] {$m_{11}$} (1);
      \draw[->,  >=stealth] (2) edge[loop above] node[midway, above] {$m_{22}$} (2);
      \draw[->,  >=stealth] (1) edge [bend left] node[midway, above] {$m_{21}$} (2);
      \draw[->,  >=stealth] (2) edge[bend left]  node[midway, below] {$m_{12}$} (1);
    \end{tikzpicture} & &
    \begin{tikzpicture}
      \node[circle, draw, minimum size=0.5cm] (1) at (0,0) {1};
      \node[circle, draw, minimum size=0.5cm] (2) at (2,0) {2};
      
      \draw[->,  >=stealth] (1) edge[loop above] node[midway, above] {$-(3\zeta)^{-1}$} (1);
      \draw[->,  >=stealth] (2) edge[loop above] node[midway, above] {$-(3\zeta)^{-1}$} (2);
      \draw[->,  >=stealth] (2) edge [bend left] node[midway, below] {$1$} (1);
    \end{tikzpicture} 
\end{tabular}
    \caption{The complete graph on two nodes and four edges corresponding to the parameters in $M$ (left).
    Example \ref{ex:twonodes} shows identifiability of $M$ up to scaling by computing the rank of $A_{\mathrm{off}}$ for a particular choice of parameters (right).}
    \label{fig::TwoCycle}
\end{figure}

\begin{equation*}
    \begin{pNiceArray}{cccc}[first-row,first-col, nullify-dots]
       & 1 \rightarrow 1    & 1 \rightarrow 2 & 2 \rightarrow 1       & 2 \rightarrow 2  \\
(12)  & \Sigma_{12} & \Sigma_{11} & \Sigma_{22} & \Sigma_{12} \\
(112)      & 2 \mathcal{K} _{112} & \mathcal{K} _{111} & 2 \mathcal{K} _{122} & \mathcal{K} _{112}  \\
(122)      & \mathcal{K} _{122} & 2 \mathcal{K} _{112} & \mathcal{K} _{222} & 2\mathcal{K} _{122}  
\end{pNiceArray} = A_{\mathrm{off}},
\end{equation*}
which we need to show has rank 3 generically. 

As argued right before the example, to prove this generic rank property we only need to exhibit one choice of parameters where the matrix has rank 3. We use the specialized trek rule in Corollary \ref{cor::simplifiedTrekRuleDAGs} and let
\begin{equation*}
    M = \begin{pmatrix}
    -1/(3 \zeta) & 1 \\ 
    0 & -1/(3 \zeta)
\end{pmatrix},
\end{equation*}
 see Figure \ref{fig::TwoCycle}. We take $\mathcal{C}_2$ and $\mathcal{C}_3$ to be the identity matrix and tensor, respectively. By Corollary \ref{cor::simplifiedTrekRuleDAGs}  we obtain the following expression for the entries of $\Sigma$ and $\mathcal{K} $

 \begin{minipage}{0.45\textwidth}
 \begin{align*}
    \Sigma_{11} & = \frac{3}{2} \zeta  & \\
    \Sigma_{12} & = \left(\frac{3}{2} \right)^2 \zeta^2 \quad  & 
    \Sigma_{22} = \frac{3}{2} \zeta + 2 \left(\frac{3}{2} \right)^3 \zeta^3 \\
 \end{align*} 
 \end{minipage}
 \hfill
 \begin{minipage}{0.45\textwidth}
 \begin{align*}
 \mathcal{K} _{111} & = \zeta & \quad \mathcal{K} _{122} & = 2 \zeta^3 \\[2mm]
 \mathcal{K} _{112} & = \zeta^2 & \quad \mathcal{K} _{222} & = \zeta + 6 \zeta^4. \\
 \end{align*}
 \end{minipage}


Plugging these into the matrix $A_{\mathrm{off}}$, its determinant becomes
$3/2 \zeta^7 + 3/4 \zeta ^4$. This is not the zero polynomial, so generically the determinant does not vanish. Thus, the parameters are generically identifiable up to a joint scaling. 

In this small 
example, the argument would work by picking just one value of $\zeta$, for example, $\zeta = 1$. However, in the full proof the number of terms in the general $\zeta$-polynomial grows rapidly, and the signs of some of the terms will be opposite. For that reason, the strategy for proving that the determinant is non-zero in general is to prove that the coefficient of the lowest degree term in $\zeta$ is non-zero.

\end{example}

\subsection{Identification when $G$ is not connected}
\label{sec:nonid}

Theorem \ref{thm::GenericIdResult} gives rise to two natural questions: is the conclusion true if 
the graph is not connected; and is the exception set $\mathcal{N}_G$ actually non-empty? We show 
the following corollary that answers both questions with the proof provided in Appendix \ref{Appendix::MainresultLemmas}. 

\begin{corollary}
    \label{cor::RankConnectedComponents}
    Let $d \geq 2$, $r \geq 3$ be integers and let $G_0 = ([d], E_0)$ be a graph with $m$ connected components and all self-loops present. Then for $(\Sigma, \mathcal{K} ) \in \mathcal{M}^{2,r}_{G_0}$,
    \begin{equation}
        A_{\mathrm{off}} = \begin{pmatrix}
            A_2(\Sigma)_{\mathrm{off}}\\ 
            A_r(\mathcal{K} )_{\mathrm{off}} 
        \end{pmatrix}
    \end{equation}
   has rank at most $d^2 - m$, and generically in $\Theta_{G_0}$ it has rank $d^2 - m$.
\end{corollary}



If a graph $G_0$ has $m \geq 2$ connected components, we can, of course,  
apply Theorem \ref{thm::GenericIdResult} for each connected component separately, 
and conclude that the parameters are generically identifiable up to component-dependent 
scaling factors. Corollary \ref{cor::RankConnectedComponents} shows that this is the 
best we can hope for, since the rank deficiency of $A_{\mathrm{off}}$ is at least the 
number of connectivity components $m$. This result is intuitively reasonable since 
each connectivity component corresponds to a subsystem of the $d$-dimensional stochastic process,
and these subsystems do not interact. Therefore, the processes can evolve
at different speeds independently of each other, and we cannot determine how fast
any of them move around from cross-sectional observations alone. 

It also follows from Corollary \ref{cor::RankConnectedComponents} that 
if $G$ is connected then $\mathcal{N}_G \neq \emptyset$. To see this, suppose that 
$$\varphi_{G,r}(\theta) = (\Sigma, \mathcal{K} ) \in \mathcal{M}^{2,r}_{G_0} \subseteq \mathcal{M}^{2,r}_{G}$$ 
for $G_0$ a subgraph of $G$ having at least two connectivity components. Then the 
corresponding $A_{\mathrm{off}}$ matrix has rank at most $d^2 - 2$ by Corollary \ref{cor::RankConnectedComponents}. This implies that $\varphi^{-1}_{G, r}(\varphi_{G,r}(\theta))$
is at least two-dimensional, and we conclude that $\varphi^{-1}_{G, r}(\mathcal{M}^{2,r}_{G_0}) \subseteq \mathcal{N}_G$. Since $\varphi^{-1}_{G, r}$ is surjective, $\mathcal{N}_G$ is non-empty. This 
shows that $\theta$ can never be globally identified up to scaling. It remains an open 
question whether $\mathcal{N}_G$ contains parameters not in 
$\varphi^{-1}_{G, r}(\mathcal{M}^{2,r}_{G_0})$ for some disconnected subgraph $G_0$. 
That is, whether there exist exceptional choices of parameters where identifiability up to scaling fails without disconnecting the graph.



\section{Estimation}
\label{sec::Estimation}

Our identification results are based on the linear equation \eqref{eq::ThmEquationSys0}. 
In this section we turn such an identification equation into an estimator 
of the drift parameter $M$. To this end, suppose $A$ is any $b \times d^2$ matrix, 
depending linearly on the cumulants, such that 
\begin{equation}
    \label{eq:id}
    \ker(A) = \mathrm{span}(\mathrm{vec}(M)).
\end{equation}
By Theorem \ref{thm::GenericIdResult} and its proof we know that \eqref{eq:id} holds for 
$\theta = (M, \mathcal{C}_2, \mathcal{C}_r) \in \Theta_G \backslash \mathcal{N}_G$ when
\begin{equation} \label{eq:standardA}
       A =
       \begin{pmatrix}
            A_2(\Sigma)_{\mathrm{off}}\\ 
            A_r(\mathcal{K} )_{\mathrm{off}} 
    \end{pmatrix}
\end{equation}
for $(\Sigma, \mathcal{K} ) = \varphi_{G, r}(\theta)$. This particular $A$-matrix has
$$
    b = \frac{d(d-1)}{2} + \binom{d+(r-1)}{r} - d \geq d^2 - 1 
$$
rows. 

For a given connected graph $G$, it follows from the proof of 
Theorem \ref{thm::GenericIdResult} that \eqref{eq:id} holds if we choose $A$ as a 
certain submatrix of \eqref{eq:standardA} containing only $d^2 - 1$ of these rows. 
On the other hand, we can also 
include additional rows from $A_{k}(\mathcal{K} )_{\mathrm{off}}$ for $k \not \in \{2, r\}$ without violating 
\eqref{eq:id}. The choice of which rows to include involves a tradeoff between computational 
complexity and statistical efficiency. For estimation purposes with a known $G$ 
we would usually also drop all columns of $A$ 
that correspond to entries of $M$ that are zero according to $G$. This would reduce $A$ to a 
$b \times |E|$ matrix. To keep the notation simple, we will here only treat estimation of 
the entire matrix $M$ corresponding to the complete graph, with obvious modifications for any 
choice of subgraph. We will, however, allow for any choice of $A$ satisfying \eqref{eq:id}.

We will from hereon assume that $M$ is normalized, that is, $\| M \|_F = 1$ where $\| \cdot \|_F$ 
denotes the Frobenius norm. 
For any specific choice of $A$ satisfying \eqref{eq:id} we define an estimator 
of $M$, which we show is consistent and asymptotically normal as the sample size tends to infinity. 
In practice we may choose $A$ given by \eqref{eq:standardA} 
with $r = 3$ or $r = 4$, or possibly a combination. Including more rows from higher order 
cumulants is computationally more costly but could potentially make the estimator more 
efficient. We report on this tradeoff in Section \ref{sec:numerical_experiments}.

\begin{figure}
\centering
\begin{tikzpicture}[scale=2,>=Latex]
	\draw[->] (-1.5,0) -- (1.5,0) node[below right] {};
	\draw[->] (0,-1.5) -- (0,1.5) node[above left] {};

	\draw[gray] (0,0) circle (1);

	\path[name path=kernel] (-1.3,-0.7) -- (1.3,0.7);
	\draw[thick,blue] (-1.3,-0.7) -- (1.3,0.7);
	\node[blue,below right] at (1.3, 0.7) {$\ker(A)$};

	\path[name path=svmin] (-1.0,-1.2) -- (1.0,1.2);
	\draw[thick,red,dashed] (-1.0,-1.2) -- (1.0,1.2);
	\node[red,above right] at (0.7,1.2) {$\mathrm{span}(v_{\min}(\hat{A}))$};

	\path[name path=unitcircle] (0,0) circle (1);
	\path[name intersections={of=kernel and unitcircle, by={k1,k2}}];
	\path[name intersections={of=svmin and unitcircle, by={s1,s2}}];
	\fill[blue] (k1) circle (1.2pt);
	\fill[red] (s1) circle (1.2pt);

    \node at (0.6,1) {$\hat{M}$};
    \node at (1.1,0.4) {$M$};

	\node[gray] at (0.85,0.7) {$\theta$};
\end{tikzpicture}
    \caption{With $v_{\min}(\hat{A})$ a right singular vector of $\hat{A}$ corresponding to its smallest singular value, the estimator $\hat{M}$ is a unit vector in $\mathrm{span}(v_{\min}(\hat{A}))$.   
    \label{fig:estimator}}
\end{figure}
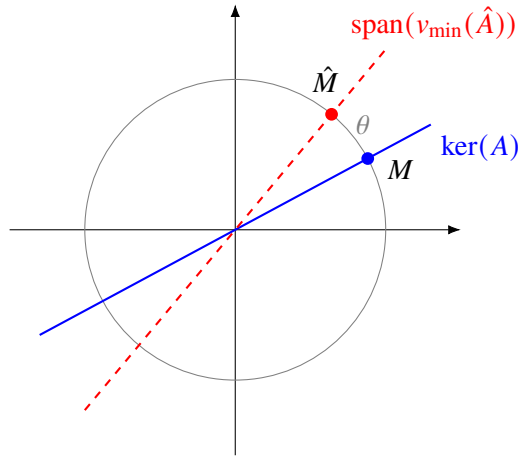

\begin{definition} \label{dfn:lsve}
Suppose $\hat{A}$ is an estimator of the $b \times d^2$ matrix $A$ with 
a unique smallest singular value. Then the least singular value estimator of $M$ is defined by 
\begin{equation}
    \label{eq:lsve}
    \mathrm{vec}(\hat{M}) = v_{\min}(\hat{A}),
\end{equation}
where $v_{\min}(\hat{A})$ denotes the right singular vector for $\hat{A}$ corresponding to its 
smallest singular value and such that $\mathrm{tr}(\hat{M}) < 0$.
\end{definition}

It is natural to think of $\mathrm{span}(v_{\min}(\hat{A}))$ as an estimator of $\ker(A)$, see 
Figure \ref{fig:estimator}. Since the right singular vector $v_{\min}(\hat{A})$ is only determined 
up to a sign, we fix its sign in Definition \ref{dfn:lsve} by requiring that $\mathrm{tr}(\hat{M}) < 0$.
This is a sensible convention since $M$ is stable with $\mathrm{tr}(M) < 0$, but it does not 
guarantee that $\hat{M}$ is stable. 

We may also observe that 
\begin{equation}
    \label{eq:lsvealt}
    \mathrm{vec}(\hat{M}) = \argmin_{v: \|v\| = 1}  \| \hat{A} v \|^2
\end{equation}
with the same convention as in Definition \ref{dfn:lsve} for fixing the sign of the minimizer.
Thus, $v_{\min}(\hat{A})$ is also an eigenvector of $\hat{A}^T \hat{A}$ corresponding to 
its smallest eigenvalue. Since the eigenvalues of a matrix are roots of the characteristic 
polynomial, the estimator is well-defined generically since the matrix $\hat{A}^T \hat{A}$ will generically have a unique smallest eigenvalue. However, least singular value 
estimation can result in a non-stable $\hat{M}$ since stability is not imposed as a 
constraint in the minimization problem \eqref{eq:lsvealt}.  

All the estimators of $A$ that we consider are linear maps of the form 
\begin{equation}
    \mathrm{vec}(\hat{A}) = \mathcal{A} \hat{\kappa}  
\end{equation}
where $\mathcal{A}$ is a fixed matrix with $b \cdot d^2$ rows, and $\hat{\kappa}$ denotes 
a unique vectorization of the empirical cumulants entering into $A$. The matrix 
$\mathcal{A}$ has integer entries and is generally extremely sparse.

Whenever the $M$-selfdecomposable distribution has finite $2k$-th-order moment, 
\begin{equation} \label{eq::AsymptoticNormalityCumulants}
    \sqrt{n}(\hat{\kappa}_n - \kappa) \xrightarrow{D} \mathcal{N} \left(0, \Omega \right)
\end{equation}
for some covariance matrix $\Omega$ by the central limit theorem. Therefore, we obtain the following theorem giving the asymptotics of the estimator proven in Appendix \ref{sec::Proofsabouttheest}. 

\begin{theorem} \label{thm::AsympoticsSingularValueEst}
Let $X_1, \ldots, X_n$ be an i.i.d. sample from a distribution in $\mathcal{P}_G$
with finite $k$-th-order moment, and suppose that $A$, whose entries are cumulants of order at most $k$, satisfies \eqref{eq:id}. Let $\hat{A}_n$ denote the plug-in estimator of $A$ based on the empirical cumulants and let $\hat{M}_n$ be the least singular value estimator 
given by \eqref{eq:lsve} in terms of $\hat{A}_n$. 
Then 
\begin{equation}
    \hat{A}_n \xrightarrow{P} A, \qquad \hat{M}_n \xrightarrow{P} M
\end{equation}
for $n \to \infty$. If, additionally, the distribution has finite $2k$-th-order moment, 
\begin{align*}
    \sqrt{n}(\mathrm{vec}(\hat{A}_n - A)) & \xrightarrow{D} \mathcal{N} \left(0, \mathcal{A} \Omega \mathcal{A}^T \right) \\
    \sqrt{n}(\mathrm{vec}(\hat{M}_n - M)) & \xrightarrow{D} \mathcal{N} 
    \left(0, \left(\mathrm{vec}(M)^T \otimes A^{+} \right) \mathcal{A} \Omega \mathcal{A}^T \left( \mathrm{vec}(M) \otimes (A^{+})^T\right) \right)
\end{align*}
for $n \to \infty$, where $A^{+}$ denotes the Moore-Penrose inverse of $A$.
\end{theorem}

To use these results in practice, we need to estimate the asymptotic covariance matrix, and for 
this we can exploit the explicit formula above to modularize the estimation. The matrix 
$\Omega$ can be estimated separately based on the empirical cumulants only. 

The following corollary, proved in Appendix \ref{sec::Proofsabouttheest}, provides 
a compact way of representing the asymptotic covariance matrix, which avoids 
explicit usage of the large matrices $\mathrm{vec}(M)^T \otimes A^{+}$ and $\mathcal{A}$. The representation uses the matrices $B_k(M)$, which are defined by equation \eqref{eq::Br(M)vec} 
in Appendix \ref{sec:lyapunov-appendix} and enter in the unique vectorizations of the 
Lyapunov equations when written as linear equations in the cumulants instead of in $M$. 

\begin{corollary} \label{cor::RewriteAsymptoticCov}
    If $A$ is given by equation \eqref{eq:standardA}, the asymptotic covariance matrix of $\mathrm{vec}(\hat{M}_n)$ can be written as
    \begin{equation*}
        (A^{+} B_{2,r}(M)) \Omega (A^{+} B_{2,r}(M))^{T},
    \end{equation*}
    where $B_{2,r}(M)$ is the block diagonal matrix with $B_2(M)$ and $B_r(M)$ on the diagonal.
\end{corollary}

If not all rows in \eqref{eq:standardA} are included in $A$, the matrix $B_{2,r}(M)$ should simply be replaced 
by the submatrix of $B_{2,r}(M)$ only indexed by the same rows as the chosen $A$.

\section{Numerical Experiments} \label{sec:numerical_experiments}

The Julia package \url{https://github.com/nielsrhansen/SteadyStateStatistics.jl} 
implements simulation from $M$-selfdecomposable distributions and computation of the 
least singular value estimator of $M$, based on the multivariate cumulant estimators 
in the Cumulants package \citep{Domino:2018}.

To illustrate properties of the estimator, we ran a simulation study with the 
Lévy process having independent and identically distributed coordinates, each of which is a 
compound Poisson process with beta distributed jumps. We parametrize the Lévy process by 
the three parameters
\begin{align*}
    \lambda > 0: & \text{ the rate parameter of the Poisson process} \\    
    \mu \in (0, 1): & \text{ the mean parameter of the beta jump distribution} \\    
    \nu > 0: & \text{ the size parameter of the beta jump distribution}. 
\end{align*}
The stable $M$-matrix is parametrized by the two parameters $\gamma \in \mathbb{R}$ and 
$\rho \in (-1/(d-1), 1)$. The parameter $\rho$ is a correlation parameter, which means 
that with $\Sigma$ the covariance matrix that solves the Lyapunov equation \eqref{eq::SecondOrderLyapunov}, then 
$$
    \rho = \frac{\Sigma_{ij}}{\sqrt{\Sigma_{ii} \Sigma_{jj}}}
$$
for all $i \neq j$. The parameter $\gamma$ controls the asymmetry of $M$ with $\gamma = 0$ 
corresponding to the symmetric $M = - c\Sigma$ for some $c > 0$. The details of this 
parametrization are given in Appendix \ref{sec:numerical_details}, which also describes 
how observations are sampled. 

We show results for the settings given by the grid of the following simulation parameters:  
\begin{align*}
    d: & \ 3, 6, 12 \\
    n: & \ 1000, 2000, 4000, 8000 \\
    \gamma: & \ 5.0, 10.0, 15.0 \\
    \rho: & \ 0.2, 0.8 
\end{align*}
and with $\lambda = 0.5$, $\mu = 0.8$ and $\nu = 1.0$. For each setting, $N = 100$ replications 
were done, and for each replication we computed the estimate $\hat{M}^{(l)}$, which by 
construction is standardized to have Frobenius norm one, that is, 
$
    \| \hat{M}^{(l)} \|_F^2 = \sum_{i,j} (\hat{M}_{ij}^{(l)})^2 = 1.
$
The matrix $M$ was also standardized to have Frobenius norm one, and the 
mean squared error was computed as
$$
    \mathrm{MSE} = \frac{1}{N} \sum_{l = 1}^N \| \hat{M}^{(l)} - M\|_F^2.
$$
Due to the standardization,  $\| \hat{M}^{(l)} - M\|_F^2 = 2(1 - \cos(\theta^{(l)}))$, 
where $\theta^{(l)}$ is the angle between the subspaces spanned by $\hat{M}^{(l)}$ and $M$, see Figure \ref{fig:estimator}. The MSE of the standardized estimator is thus 
equal to the mean of twice the cosine distance between $\hat{M}^{(l)}$ and $M$.

The MSE can be decomposed into a total variance term and a total squared bias term,
$$
\mathrm{MSE} = \frac{1}{N} \sum_{l = 1}^N \| \hat{M}^{(l)} - \overline{M}\|_F^2
+ \| \overline{M} - M\|_F^2.
$$
According to Theorem \ref{thm::AsympoticsSingularValueEst}, the $\sqrt{n}$-scaled 
total bias, $\sqrt{n} \| \overline{M} - M\|_F$, should vanish with $n$ tending to 
infinity, and the $\sqrt{n}$-scaled
root mean squared error should tend to the square root of the asymptotic total variance, 
that is, the square root of the trace of the asymptotic covariance matrix given by  
Theorem \ref{thm::AsympoticsSingularValueEst} or Corollary \ref{cor::RewriteAsymptoticCov}.

\begin{figure}
    \centering
    \includegraphics[width=\linewidth]{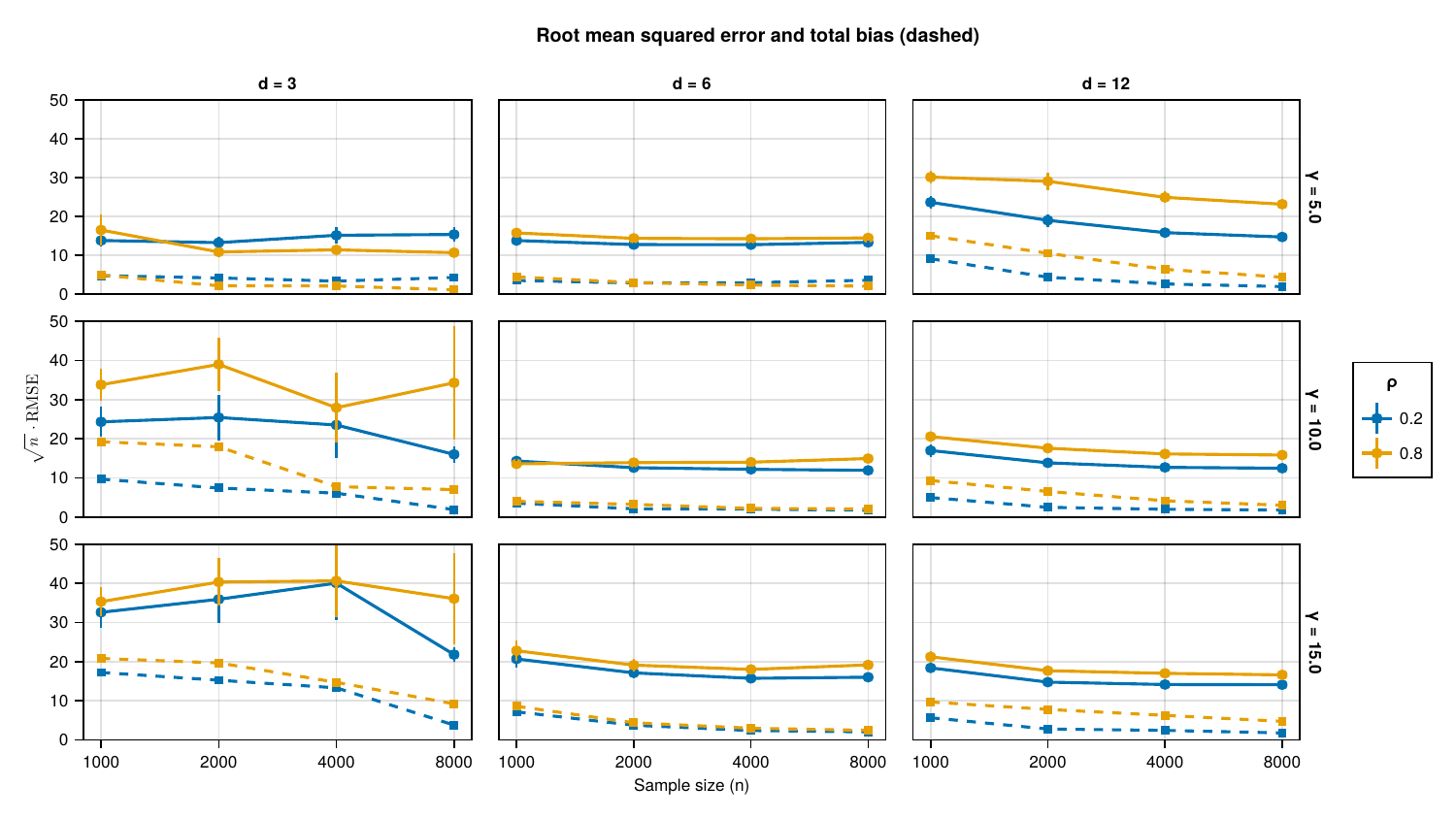}
    \caption{Estimation errors for the least singular value estimator based on all second- and third-order cumulants. Circles connected with full lines show the $\sqrt{n}$-scaled root mean squared error, while the squares connected with dashed lines show the $\sqrt{n}$-scaled total bias.}
    \label{fig:estimation_errors}
\end{figure}

Figure \ref{fig:estimation_errors} shows the scaled root mean squared error
as well as the scaled total bias as a function of the sample size for 
the different dimensions and different values of the simulation parameters. 
The estimator is the least singular value estimator based on all second- 
and third-order cumulants. We see that for several of the settings there is 
a notable finite sample bias, in particular when the correlation is set to 
the high value $\rho = 0.8$, but for large enough sample sizes the bias 
disappears and the scaled root mean squared error stabilizes. 

Figure \ref{fig:estimation_errors_asymp} compares the scaled root mean squared error
to the square root of the asymptotic total variance obtained by 
Corollary \ref{cor::RewriteAsymptoticCov}. For all settings presented here,
the asymptotic total variance eventually concurs with the mean squared error 
when the sample size is large enough. This is aligned with the observation
that the total bias vanishes.

\begin{figure}
    \centering
    \includegraphics[width=\linewidth]{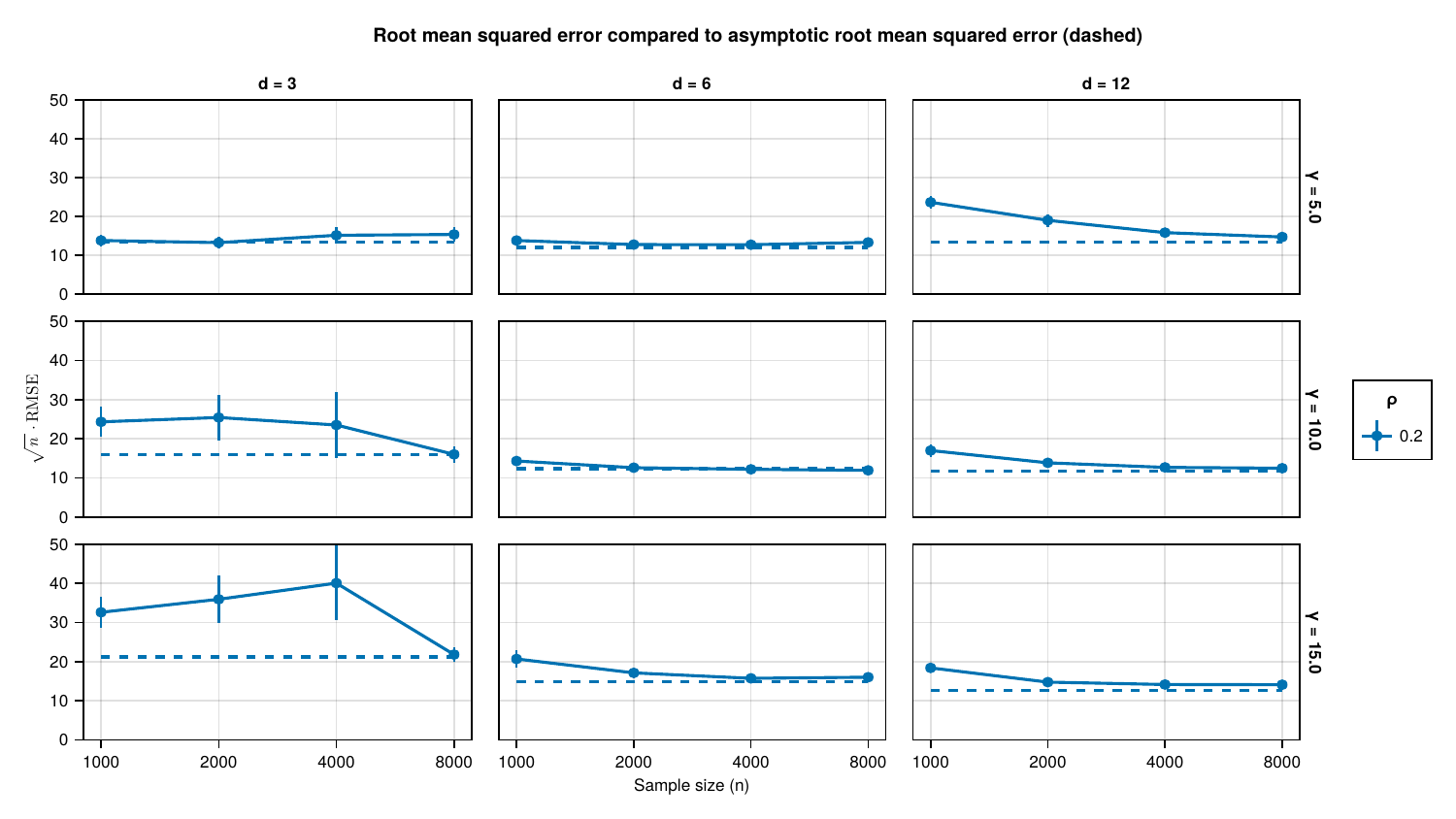}
    \caption{Root mean squared error scaled with $\sqrt{n}$ compared to 
    the square root of the total asymptotic variance (dashed).}
    \label{fig:estimation_errors_asymp}
\end{figure}

We ran simulations for other parameter values (data not shown). For $d = 24$ or
$\gamma = 0$, the estimation problem is harder, and even larger sample sizes are 
required for accurate estimation. Other choices of parameters for the compound 
Poisson process were also tried, e.g., $\lambda = 0.1, 2.0$, $\mu = 0.2$ and 
$\nu = 0.1, 10$. While the results  differed quantitatively from the results 
shown, the qualitative differences were minor. 

In addition to the estimator based on second- and third-order cumulants only, 
we implemented and tested the estimator that additionally included the fourth-order 
cumulants. 
Figure \ref{fig:estimation_errors_k4} in Appendix \ref{sec:numerical_details} 
shows the results for this least singular value estimator based on all second-, third-  
and fourth-order cumulants. For $\rho = 0.8$, inclusion of 
fourth-order cumulants generally increased the estimation error. For $\rho = 0.2$, 
the effect of including the fourth-order cumulants was mixed and depended on $d$ and $\gamma$. 

\section{Discussion} \label{sec::discussion} 

Our main identifiability result, Theorem \ref{thm::GenericIdResult}, leaves two minor questions 
open. First, it remains unknown to what extent the assumption that the graph 
contains all self-loops can be relaxed. Our proof strategy relies on the self-loops
to ensure the existence of a sufficiently simple and stable $M$ that allows for 
explicit computations, see also Remark \ref{rem::Selfloops}. Second, it remains unknown whether the exceptional set, where 
identification fails, corresponds only to parameter values that disconnect the graph. 
If we are willing to make additional assumptions, stronger results can also be obtained. 
If we, for example, additionally assume that the $r$-th-order cumulant 
of the Lévy process, for some $r \geq 3$, is known and diagonal, Proposition \ref{prop::ThirdOrderID} 
in Appendix \ref{Appendix::AddtionalIDResults} shows that the drift matrix $M$ is, 
in fact, generically identifiable for \emph{any} directed graph with all self-loops present.

There are also natural generalizations and extensions of our results. One direction of
future research is to allow the cumulants of the Lévy process to be non-diagonal,
as encoded by blunt edges in Appendix \ref{Appendix::treks}. Following 
the arguments by \cite{varando20a}, the marginalization of an $M$-selfdecomposable 
distribution is a distribution with cumulants solving continuous Lyapunov equations 
for a new $M$ and a new $\mathcal{C}$, whose sparsity patterns are obtained by 
projecting the directed graph to a graph with directed and blunt edges. This generalization
would be a way of expanding the model class to be able to handle latent variables. 
However, the marginalized $M$ is not guaranteed to be stable, and the marginalized
$\mathcal{C}_2$ is, for example, not guaranteed to be positive definite \citep{varando20a}. 
The proof of Theorem \ref{thm::GenericIdResult} could, nevertheless, easily be modified
to allow for specific $r$-th-order cumulants to be non-zero, e.g., those whose indices
do not correspond to the rows used to prove the generic rank of $A_{\mathrm{off}}$. Another 
direction of future research is to consider structural identifiability instead. The results 
by \cite{amendola2025structuralidentifiabilitygraphicalcontinuous} on structural identifiability 
are based on the second order continuous Lyapunov equation only, and by expanding their 
approach to take higher order cumulants into account, we could likely answer some of the 
open questions by \cite{amendola2025structuralidentifiabilitygraphicalcontinuous}. 

The present paper focuses on the identification theory and does not consider specific 
applications, but the theory is motivated by the possible applications to cross-sectional 
observations from single cell biology. The $M$-selfdecomposable distributions can, as 
steady-state distributions of solutions to stochastic differential equations, be given 
natural causal interpretations \citep{Sokol:2014, varando20a}, which was exploited by, for example, 
\cite{Wang:2023aa} to learn gene regulatory networks from temporal snapshots. 
The network is then encoded by the sparsity pattern of $M$, and if the model has a 
causal interpretation, an edge represents a direct causal effect. An alternative to 
$M$-selfdecomposable distributions is classical linear structural equations
\citep{DrtonOversigt}, but feedback loops present both mathematical and interpretational 
challenges for such models. The implicit temporal nature of $M$-selfdecomposable distributions
allows feedback loops to be naturally incorporated into the model, and the 
interpretation and properties of the model do not depend on whether there are feedback 
loops or not. However, a realistic application to single cell data requires a measurement 
noise model, such as the zero-inflated Poisson model by \cite{lorch_latent_2026}, and this is 
beyond the scope of the current paper.

Finally, as demonstrated by our simulation study, even when the drift matrix 
is identifiable in the non-Gaussian case from the higher-order cumulants, estimation is 
hard for small sample sizes. This is also seen in the general literature that uses 
non-Gaussianity and higher-order cumulants for identification and estimation 
\citep{JMLRv7shimizu06a, schkoda2024causaldiscoverylinearnongaussian, tramontano2025causaleffectidentificationlvlingam}, and \cite{JMLR:Wang_Drton} explicitly 
comments on the need for large sample sizes. An ongoing research project explores how
additional assumptions, e.g., sparsity of $M$ or homogeneity of the Lévy process across 
dimensions, can improve finite sample estimation accuracy.

\section*{Acknowledgment}

We would like to thank Mathias Drton and Anton Rask Lundborg for helpful discussions. This work was supported by a research grant (NNF20OC0062897) from Novo Nordisk Fonden.
Proposition \ref{prop::tensor-lyapunov}, that the cumulants of the steady-state distribution satisfy the higher-order continuous Lyapunov equations, was first proved by Jeffrey Adams and included in his PhD thesis \citep{Jeff:2024}. The proof given here was obtained independently.

\newpage

\appendix
\appendixone

\section{Appendix -- proofs of main results}
\subsection{The Lyapunov equations} \label{sec:lyapunov-appendix}

\begin{proof}[of Proposition~\ref{prop::tensor-lyapunov}] When $M$ is a 
stable matrix and $Z$ is a Lévy process satisfying \eqref{eq:levy-int-cond}, 
the stochastic integral 
\begin{equation}
		\label{eq:MSD}
		X = \int_0^{\infty} e^{sM} \mathrm{d} Z_s
\end{equation}
is well defined, and the distribution of $X$ is the unique steady-state distribution 
in $\mathcal{P}_G$ determined by $M$ and the Lévy process. This was known by \cite{Sato:1984}, 
and stated indirectly in their introduction with reference to an abstract by \cite{Wolfe:1982}.
It is a direct consequence of Theorem 5.2(i) by \cite{sato_stationary_1983}.

Using the integral representation \eqref{eq:MSD} and multilinearity of the cumulant operator,
	\begin{align*}
		\mathcal{K}  & = \mathrm{cum}_k(X)                                                                                                                                                              \\
		  & = \int_0^{\infty} \! \cdots \! \int_0^{\infty}  \mathrm{cum}_k(\mathrm{d} Z_{s_1}, \ldots, \mathrm{d} Z_{s_k}) \times_1 e^{s_1 M} \times_2 e^{s_2 M} \cdots \times_{k} e^{s_k M} \\ & = \int_0^{\infty} \mathcal{C}_k \times_1 e^{s M} \times_2 e^{sM} \cdots \times_{k} e^{s M} \mathrm{d} s
	\end{align*}
	where we have used that for a Lévy process, the $k$-th-order cumulant measure is ``diagonal'' and equals $$\mathrm{cum}_k(\mathrm{d} Z_{s_1}, \ldots, \mathrm{d} Z_{s_k}) = \mathcal{C}_k \delta_{s_1, \ldots, s_k} H^{1}(ds_1, \ldots, ds_k)$$ with $H^{1}$ the $1$-dimensional Hausdorff measure.
	That $\mathcal{K} $ solves~\eqref{eq:k-lyapunov} follows from Proposition \ref{prop::XuResult}.
\end{proof}


For the second order Lyapunov equation it is well-known that it can be vectorized as 
$$(I \otimes M + M \otimes I) \mathrm{vec}(\Sigma) + \mathrm{vec}(\mathcal{C}_k) = 0,$$ 
which shows that it has a unique solution for $\Sigma$ if no pair of eigenvalues of $M$ add up to zero. This can be generalized to the continuous Lyapunov equation of any order.
\begin{proposition} \label{prop::VecOfEquation}
    The vectorization of the $k$-th-order continuous Lyapunov equation, \eqref{eq:k-lyapunov}, can be written as 
    \begin{equation} \label{eq::vecKVecofKLyapunov}
        0 = \mathrm{vec}(\mathcal{C}_k) + \left( \sum_{i = 1}^k (\underbrace{I_{d}\otimes\cdots\otimes I_{d}}_{k-i \text{ times}} \otimes M \otimes  \underbrace{I_{d}\otimes\cdots\otimes I_{d}}_{i-1 \text{ times}})   \right) \mathrm{vec}(\mathcal{K} ).
    \end{equation}
    As a consequence of this, the $k$-th-order continuous Lyapunov equation is uniquely solvable whenever no $k$ eigenvalues (the same eigenvalue can be picked multiple times) add up to zero.
\end{proposition}
\begin{proof}
    By rules of vectorizations and Kronecker products, see Chapter 2 by \cite{magnus_neudecker_2019}, we can write
    \begin{equation*}
        \mathrm{vec}(\mathcal{K}  \times_i M) = (\underbrace{I_{d}\otimes\cdots\otimes I_{d}}_{k-i \text{ times}} \otimes M \otimes  \underbrace{I_{d}\otimes\cdots\otimes I_{d}}_{i-1 \text{ times}})\mathrm{vec}(\mathcal{K} ),
    \end{equation*}
    for $\mathcal{K}  \in \mathrm{Sym}^k(\mathbb{R}^d)$ and $M \in \mathbb{R}^{d \times d}$. This yields equation \eqref{eq::vecKVecofKLyapunov}. 

    Furthermore, since the eigenvalues of $A \otimes B$ are equal to products of the eigenvalues of $A$ and $B$, it follows that no matter the place of the single $M$, the eigenvalues of any tensor product, $I_d \otimes \cdots \otimes I_d \otimes M \otimes I_d \otimes \cdots \otimes I_d $, is just equal to the eigenvalues of $M$ (each with multiplicity $d$). Furthermore, since all the matrices of this form, $I_d \otimes \cdots \otimes I_d \otimes M \otimes I_d \otimes \cdots \otimes I_d $, commute with each other (since all the $M$'s are placed in different modes), the eigenvalues of the sum of them in equation \eqref{eq::vecKVecofKLyapunov} is the sum of the their eigenvalues. Thus, any eigenvalue of
    the sum in \eqref{eq::vecKVecofKLyapunov} is a sum of $k$ eigenvalues of $M$. And if no $k$ eigenvalues of $M$ add up to zero, the sum is a matrix of full rank, and the $k$-th-order continuous Lyapunov equation is uniquely solvable for $\mathcal{K} $ given $M$ and $\mathcal{C}_k$. 
\end{proof}

Equation \eqref{eq::vecKVecofKLyapunov} is the standard vectorization of the Lyapunov equation, and 
we could instead do the unique vectorization of both the equation and of $\mathcal{K} $,
\begin{equation} \label{eq::Br(M)vec}
    0 = \mathrm{vec}_u(\mathcal{C}_k) + B_k(M) \mathrm{vec}_u(\mathcal{K} ).
\end{equation}
The rows of $B_k(M)$ can be obtained by adding the columns of the matrix in equation \eqref{eq::vecKVecofKLyapunov} corresponding to the same cumulant, i.e., for an index set $(i_1 \dots i_k)$ taking all columns corresponding to an index of $\mathcal{K} $ which is a permutation of $(i_1 \dots i_k)$. To also have it be the unique vectorization of the equation we only keep the rows indexed by $(i_1 \dots i_k)$ such that $i_1 \leq \cdots \leq i_k$ as was done for the unique vectorization isolating $\mathrm{vec}(M)$ instead in equation \eqref{eq::VecThirdOrder}.

When $M$ is stable, there is an explicit formula for the solution of the continuous Lyapunov equation.
\begin{proposition}[Corollary 3.1, \cite{Xu:2021}] \label{prop::XuResult} 
Let $M$ be a  $d \times d$ stable matrix, then the $k$-th-order continuous Lyapunov equation \eqref{eq:k-lyapunov} has the unique solution
\begin{equation}
    \label{eq::generalSol}
    \mathcal{K}  = \int_{0}^{\infty} \mathcal{C}_k \times_1 e^{M t} \times_2 e^{M t} \cdots \times_{k} e^{M t} \mathrm{d} t. 
\end{equation}
\end{proposition}
Equation \eqref{eq::generalSol} is the key to deriving the so-called trek formulas for the entries of the cumulants of the continuous Lyapunov equations. The general trek rules provide a graphical description of the entries of the cumulants that can allow one to determine zeros in the cumulants just from the absence of certain treks, corresponding to the absence of several paths, in the graph. Additionally, 
a specialized version of the trek formulas is a key ingredient in the proof of the main identifiability result, Theorem \ref{thm::GenericIdResult}. The details and proofs concerning treks and the trek representations for cumulants are provided in Appendix \ref{Appendix::treks}.

\subsection{Proof of the main result and additional lemmas and corollaries} \label{Appendix::MainresultLemmas}

The proof of Theorem \ref{thm::GenericIdResult} uses a specific restricted trek rule for the second- and $r$-th-order cumulants, very similar to the restricted trek rule given in Proposition 4.3 by \cite{boege2024conditionalindependencestationarydiffusions}. Such restricted trek rules can be seen as special cases of a more general trek rule first published by \cite{hansen2024trekrulelyapunovequation} in the covariance case. We provide the extension of the general trek rule from the covariance case to 
cumulants of any order in Appendix \ref{Appendix::treks}, where a definition of treks is also given. 
For the purpose of proving Theorem~\ref{thm::GenericIdResult}, we give here a simple proof of the 
restricted trek rule for directed acyclic graphs. 

\begin{lemma}
    \label{lem::simplifiedtrekruleDAG}
    Let $G = ([d], E)$ be any directed acyclic graph (DAG) with all self-loops and with the nodes ordered in a topological order. Let $r \geq 3$ be an integer and consider the second- and $r$-th-order Lyapunov model with  $\mathcal{C}_2$ and $\mathcal{C}_r$ equal to the identity matrix and identity tensor, respectively, and $M \in \mathbb{R}^{E}$ such that $M_{ii} = -1/(r \zeta)$. Then we have the following specialized trek rules for the second- and $r$-th-order cumulants
    \begin{align*}
        \Sigma_{i_1 i_2} &=  \sum_{ \tau \in \mathcal{T}(i_1, i_2)} \zeta^{\ell_{i_1} + \ell_{i_2} + 1}  \left(\frac{r}{2} \right )^{\ell_{i_1} + \ell_{i_2} + 1}  \binom{\ell_{i_1} + \ell_{i_2}}{\ell_{i_1}} \omega(M, \tau) \\ 
        \mathcal{K} _{i_1 \dots i_r} &= \sum_{ \tau \in \mathcal{T}(i_1, \dots, i_r)}  \zeta^{\sum_{j = 1}^r\ell_{i_j} +1} \frac{(\sum_{j = 1}^r\ell_{i_j} )!}{\prod_{j = 1}^r (\ell_{i_j}!) } \omega(M, \tau),
    \end{align*}
    where $\mathcal{T}(i_1, i_2)$ and $\mathcal{T}(i_1, \dots, i_r)$ denote all treks \textit{without} self-loops between $i_1$ and $i_2$ and $i_1, \dots, i_r$, respectively, and 
    $\omega(M,\tau)$ denotes the trek monomial. That is, $\omega(M, \tau) = \prod_{ \alpha \rightarrow \beta \in \tau} M_{\beta \alpha}$ since $\mathcal{C}_2$ and $\mathcal{C}_r$ are the identity.  
\end{lemma}
\begin{proof}
    For notational convenience the proof is given for the case $r = 3$. It is, however, 
    easily seen from this proof how the general case follows by keeping track of 
    $r$ mode products instead of $3$, and replacing $3$ by $r$ everywhere.  
    
    Let $M_{\mathrm{DAG}} = M - (-1/(3 \zeta) I_{d\times d})$ denote the strictly lower triangular part of $M$ which corresponds to all the edges in $G$ which are not self-loops, by subtracting a scaling of the $d \times d$ identity matrix, $I_{d \times d}$. We decompose $M$ as
    \begin{align*}
        M = M_{\mathrm{DAG}}-\frac{1}{3 \zeta} I_{d \times d}.  
    \end{align*}
    Since $M_{\mathrm{DAG}}$ is lower-triangular it has spectral radius $0$, which is 
    especially less than $1$. We can use this decomposition of $M$ instead of the $(\Lambda - I_{d \times d})$ used in Proposition \ref{prop::LambdaTrekRuleGeneral}, or in Proposition 2.4 by \cite{hansen2024trekrulelyapunovequation} for the covariance case, and do exactly the same proof with the identity instead of $\mathcal{C}$ and $M_{\mathrm{DAG}}$ instead of $\Lambda$ and the exponent of $\mathrm{e}$ will be $-2/(3 \zeta)$ and $-1/\zeta$, respectively. When letting $I_{d \times d \times d}$ denote the identity $d \times d \times d$ tensor, writing out the derivation of the third-order trek formula in this special case becomes
        \begin{align*}
        \mathcal{K}  &= \int_{0}^{\infty} I_{d \times d \times d} \times_1 e^{M t} \times_2 e^{M t} \times_{3} e^{M t} \mathrm{d} t \\ 
        &= \int_{0}^{\infty} I_{d \times d \times d} \times_1 e^{(M_{\mathrm{DAG}}-1/(3 \zeta) I_{d \times d}) t} \times_2 e^{(M_{\mathrm{DAG}}-1/(3 \zeta) I_{d \times d}) t}  \times_{3} e^{(M_{\mathrm{DAG}}-1/(3 \zeta) I_{d \times d}) t} \mathrm{d} t \\ 
        &= \int_{0}^{\infty} e^{- \frac{1}{\zeta} t }  I_{d \times d \times d} \times_1 e^{M_{\mathrm{DAG}} t} \times_2 e^{M_{\mathrm{DAG}} t}  \times_{3} e^{M_{\mathrm{DAG}} t} \mathrm{d} t \\
        &= \int_{0}^{\infty} e^{- \frac{1}{\zeta} t }  I_{d \times d \times d} \times_1 \sum_{n = 0}^{\infty} \frac{t^n}{n!} M_{\mathrm{DAG}}^n \times_2 
        \sum_{m = 0}^{\infty} \frac{t^m}{m!} M_{\mathrm{DAG}}^m \times_3  \sum_{k = 0}^{\infty} \frac{t^k}{k!} M_{\mathrm{DAG}}^k \mathrm{d} t \\ 
        &= \sum_{n = 0}^{\infty} \sum_{m = 0}^{\infty} \sum_{k = 0}^{\infty} \int_{0}^{\infty} \frac{t^{n+m+k} e^{- \frac{1}{\zeta} t }}{n! m! k!}  
        I_{d \times d \times d} \times_1  M_{\mathrm{DAG}}^n \times_2 M_{\mathrm{DAG}}^m \times_3  M_{\mathrm{DAG}}^k \mathrm{d} t \\ 
        &=  \sum_{n = 0}^{\infty} \sum_{m = 0}^{\infty} \sum_{k = 0}^{\infty} \frac{\zeta^{n+m+k+1}  \Gamma(n+m+k+1)}{n! m! k!} 
        I_{d \times d \times d} \times_1  M_{\mathrm{DAG}}^n \times_2 M_{\mathrm{DAG}}^m \times_3  M_{\mathrm{DAG}}^k. 
    \end{align*}
The $(i,j,k)$ entry of the tensor $I_{d \times d \times d} \times_1  M_{\mathrm{DAG}}^n \times_2 M_{\mathrm{DAG}}^m \times_3  M_{\mathrm{DAG}}^k$ corresponds to the sum of all treks in the acyclic part between $i$, $j$ and $k$ with a given top node with the paths going to each of them having lengths $n,m,k$, respectively. To see this we write out the $(i,j,k)$ entry 
\begin{align*}
    \left(I_{d \times d \times d} \times_1  M^n_{\mathrm{DAG}} \times_2 M^m_{\mathrm{DAG}} \times_3  M^k_{\mathrm{DAG}}\right)_{i j k} &= \sum_{\alpha = 1}^{d} 
     (M^{n}_{\mathrm{DAG}})_{i \alpha} \cdot (M^{m}_{\mathrm{DAG}})_{j \alpha} \cdot (M^{k}_{\mathrm{DAG}})_{k \alpha},
\end{align*} 
and note that $(M^n_{\mathrm{DAG}})_{i \alpha}$ is exactly the sum of all directed walk polynomials of length $n$ from $\alpha$ to $i$ in the directed acyclic part of $G$ since $M_{\mathrm{DAG}}$ is the adjacency matrix (with weights) of the directed acyclic part of $G$. See also 
the proof of Proposition \ref{prop::ThirdOrderTrekwithS}. 

Because the graph is assumed to be a DAG, the sum above becomes finite, and for the $(i,j,k)$-th entry we therefore obtain exactly the finite sum over all treks between $i$, $j$ and $k$ as given the statement of the proposition by inserting the above into the sum. 

The proof is completely analogous for the covariance just with one term less and the exponent of the exponential being $(-2/(3 \zeta)) t$ instead of $(- 1/ \zeta) t$. 
\end{proof}

The specific trek rule above gives a finite sum representation of the entries of the cumulants, as opposed to the potential infinite sums in the general trek rules in Appendix \ref{Appendix::treks}.
The proof of Theorem~\ref{thm::GenericIdResult} uses the special case given in the following 
corollary.

\begin{corollary}
\label{cor::simplifiedTrekRuleDAGs}
        If we further assume that $M_{\beta \alpha} = 1$ for $\alpha \rightarrow \beta \in E$ with $\beta \neq \alpha$, then 
  \begin{align*}
        \Sigma_{i_1 i_2} &=  \sum_{ \tau \in \mathcal{T}(i_1, i_2)} \zeta^{\ell_{i_1} + \ell_{i_2} + 1}  \left(\frac{r}{2} \right )^{\ell_{i_1} + \ell_{i_2} + 1}  \binom{\ell_{i_1} + \ell_{i_2}}{\ell_{i_1}}\\ 
        \mathcal{K} _{i_1 \dots i_r} &= \sum_{ \tau \in \mathcal{T}(i_1, \dots, i_r)}  \zeta^{\sum_{j = 1}^r\ell_{i_j} +1} \frac{(\sum_{j = 1}^r\ell_{i_j} )!}{\prod_{j = 1}^r(\ell_{i_j}!) }.
    \end{align*}
\end{corollary}

In addition to the specific trek rule, the proof of Theorem~\ref{thm::GenericIdResult} 
also relies on the following sum identity.

\begin{lemma}
\label{lem::forsum}
    Let $d,q$ be non-negative integers such that $q \leq d-1$ and $r \geq 3$ an integer, then 
    \begin{equation}
    \label{eq::idforsumfordeterminant}
            \sum_{i = 0}^{d-1} \left(\frac{r}{2} \right)^{d-1-i} (-1)^{i} c(d,q,i)  = \left(\frac{r}{2} -1 \right)^{d-1} (-1)^{q},
    \end{equation}
where the coefficients are $$c(d,q,i) = \sum_{j = 0}^{i } (r-1)^{j} \binom{q}{j} \binom{d-1-q}{i-j}.$$    
\end{lemma}
\begin{proof}[]
We first prove the following recursive identity for the coefficients
\begin{equation*}
    c(d,q,i) = c(d, q-1, i) + (r-2)c(d-1, q-1, i-1). 
\end{equation*}

To prove this, note that $c(d,q,i)$ is the coefficient in front of $x^{i}$ if you expand the following polynomial in $x$ using the binomial theorem, $(1+(r-1)x)^{q} (1+x)^{d-1-q}$. Thus, $c(d,q-1,i)$ is the coefficient in front of $x^{i}$ in $(1+(r-1)x)^{q-1} (1+x)^{d-1-(q-1)}$ and $c(d-1, q-1, i-1)$ is the coefficient in front of $x^{i-1}$ in $(1+(r-1)x)^{q-1} (1+x)^{d-2-(q-1)}$. Hence, $c(d, q-1, i) + c(d-1, q-1, i-1)$ will be the coefficient in front of $x^{i}$ in 
\begin{align*}
    &(1+(r-1)x)^{q-1} (1+x)^{d-1-(q-1)} + (r-2)x (1+(r-1)x)^{q-1} (1+x)^{d-2-(q-1)}  \\
    &=  (1+(r-1)x)^{q-1}(1+x)^{d-2-(q-1)} (1+x +(r-2)x) \\
    &=  (1+(r-1)x)^{q} (1+x)^{d-1-q}. 
\end{align*}
Thus, the recursion formula is true. 

We now obtain the result about the summation by induction on $(d,q)$, where we induct over this set by what the sum $d + q$ is. 

The base case is $d = 2$ and $q = 0$. The proof is the same for any $d$, so we just prove the formula for any $d$ and $q = 0$. When $q$ is 0 we obtain 
\begin{align*}
    c(d,0,i) = \sum_{j = 0}^{i } (r-1)^{j} \binom{0}{j} \binom{d-1}{i-j} = \binom{d-1}{i},
\end{align*}
since the only non-zero contribution to the sum above is when $j = 0$. Inserting this into the left hand side of equation \eqref{eq::idforsumfordeterminant} we can prove the base case by application of the binomial formula
\begin{align*}
    \sum_{i = 0}^{d-1} \left(\frac{r}{2} \right)^{d-1-i} (-1)^{i} c(d,0,i) = \sum_{i = 0}^{d-1} \left(\frac{r}{2} \right)^{d-1-i} (-1)^{i} \binom{d-1}{i} = \left( \frac{r}{2} -1 \right)^{d-1}.  
\end{align*}


We now consider the induction step, so consider the lefthand side in equation \eqref{eq::idforsumfordeterminant} and apply the recursion formula to obtain
\begin{align*}
    &\sum_{i = 0}^{d-1} \left(\frac{r}{2} \right)^{d-1-i} (-1)^{i} c(d,q,i) \\ &= 
    \sum_{i = 0}^{d-1} \left(\frac{r}{2} \right)^{d-1-i} (-1)^{i} c(d,q-1,i) +
    (r-2)\sum_{i = 0}^{d-1} \left(\frac{r}{2} \right)^{d-1-i} (-1)^{i} c(d-1,q-1,i-1).  
\end{align*}
By the induction hypothesis the left sum is equal to $\left( r/2 - 1 \right)^{d-1} (-1)^{q-1}$. The rightmost sum can be written as
\begin{align*}
    &\sum_{i = 0}^{d-1} \left(\frac{r}{2} \right)^{d-1-i} (-1)^{i} c(d-1,q-1,i-1) \\ &= - \sum_{i = 0}^{d-1} \left(\frac{r}{2} \right)^{(d-2)-(i-1)} (-1)^{i-1} c(d-1,q-1,i-1) \\
    &= - \sum_{i = -1}^{d-2} \left(\frac{r}{2} \right)^{(d-2)-i} (-1)^{i} c(d-1,q-1,i) \\
    &= - \sum_{i = 0}^{d-2} \left(\frac{r}{2} \right)^{(d-2)-i} (-1)^{i} c(d-1,q-1,i) 
    - \left(\frac{r}{2} \right)^{d-1} c(d-1, q-1, -1) \\
    &= - \left(\frac{r}{2} -1 \right)^{d-2} (-1)^{q-1},
\end{align*}
by using the induction hypothesis and that $c(d-1, q-1, -1) = 0$. Inserting the expressions for the two sums we obtain
\begin{align*}
    \sum_{i = 0}^{d-1} \left(\frac{r}{2} \right)^{d-1-i} (-1)^{i} c(d,q,i) &= \left(\frac{r}{2} -1 \right)^{d-1} (-1)^{q-1} - (r-2)\left(\frac{r}{2} -1 \right)^{d-2} (-1)^{q-1} \\
    &=  \left(\frac{r}{2} -1 \right)^{d-1} (-1)^{q} \left( -1 + \frac{r-2}{r/2-1} \right) \\
    &=  \left(\frac{r}{2} -1 \right)^{d-1} (-1)^{q} \left(\frac{-r/2 + 1 + r-2}{r/2-1} \right) = 
    \left(\frac{r}{2} -1 \right)^{d-1} (-1)^{q},
\end{align*}
as we needed to show. 
\end{proof}

In the proof of Theorem \ref{thm::GenericIdResult} we will only need the general expression for the rows of $A_r(\mathcal{K} )$ when precisely two distinct subscripts appear, namely the rows $(i_1 i_2)$ and $(i_1 \dots  i_1 i_2)$ (and $(i_1  i_2 \dots i_2)$),  when $i_1 \neq i_2$. The explicit expression for these rows is

\begin{align}
\label{eq::A3KEntries}
        A_r(\mathcal{K} )_{(i_1 \dots i_1 i_2), (\alpha \rightarrow \beta)} = \begin{cases}
        0 & \mathrm{ if } \beta  \neq i_1, i_2\\
         (r-1) \mathcal{K} _{i_1 \dots i_1 i_2 \alpha} & \mathrm{ if } \beta = i_1\\
        \mathcal{K} _{i_1 \dots i_1 \alpha } & \mathrm{ if } \beta = i_2. \\
    \end{cases}
\end{align}
However, as discussed in Section \ref{sec:proof-reorganize}, the general expressions can easily be obtained as well. 

\begin{proof}[of Theorem \ref{thm::GenericIdResult}]

    Since we only consider diagonal $\mathcal{C}_2$ and $\mathcal{C}_r$, the linear system corresponding to the second and $r$-th-order Lyapunov equations from equation \eqref{eq::BigLinSys} reduces to 
    \begin{equation}
    \label{eq::ThmEquationSys}
        \begin{pmatrix}
            A_2(\Sigma)_{\mathrm{off}} & 0 & 0 \\ 
            A_r(\mathcal{K} )_{\mathrm{off}} & 0 & 0 \\
            A_2(\Sigma)_{\mathrm{diag}} & I_{d} & 0 \\
            A_r(\mathcal{K} )_{\mathrm{diag}} & 0 & I_{d} \\ 
        \end{pmatrix} 
        \begin{pmatrix}
            \mathrm{vec}(M) \\
            \mathrm{diag}(\mathcal{C}_2 )\\
             \mathrm{diag}(\mathcal{C}_r)\\
        \end{pmatrix}
         = 0,
    \end{equation}
    where $A_2(\Sigma)_{\mathrm{off}}$ and $A_r(\mathcal{K} )_{\mathrm{off}}$ denote all the rows indexed by the off-diagonal tensor entries of $A_2(\Sigma)$ and $A_r(\mathcal{K} )$, respectively. That is all the rows not indexed by $(ii)$ or $(i \dots i)$, and  $A_2(\Sigma)_{\mathrm{diag}}$ and $A_r(\mathcal{K} )_{\mathrm{diag}}$ denote the rows indexed by the diagonal entries. 
    
    Thus, the question of identifiability, i.e., whether this matrix in equation \eqref{eq::ThmEquationSys} has full rank minus 1, reduces to the question of whether the matrix  
    $$\begin{pmatrix}
            A_2(\Sigma)_{\mathrm{off}}\\ 
            A_r(\mathcal{K} )_{\mathrm{off}} 
        \end{pmatrix}$$ 
        has rank $d^2 -1$. However, this matrix is in general not square, 
        so we have to consider the determinant of a $(d^2 -1 ) \times (d^2 -1)$ submatrix. 

Any connected graph will have at least one polytree as a subgraph without the self-loops, which is a connected directed acyclic graph with exactly $d-1$ edges. Pick any such subpolytree and let 
 $G' = ([d], E')$ be the subgraph of $G$ where we have the edges in the polytree as well as all the self-loops. Thus, $G'$ will be a connected directed acyclic graph (disregarding the self-loops) with $d-1$ non-self-loop edges and all self-loops and be a subgraph of $G$. Firstly, we can without loss of generality assume that the nodes of $G'$ are ordered according to a topological order, $\leq$. We order any possible edges on $d$ nodes 
        according to a lexicographic ordering with the ordering on each coordinate the topological order, thus $(i \rightarrow j) \leq (i' \rightarrow j')$ if $i \leq i'$ or $i = i'$ and $j \leq j'$. According to this order the first element is $1 \rightarrow 1$, this is the column we will omit from the $(d^2-1) \times (d^2-1)$ submatrix we consider.

        We pick the following rows. We pick all of $A_2(\Sigma)_{\mathrm{off}}$ and of $A_r(\mathcal{K} )_{\mathrm{off}}$ we pick the rows of the form $(ij \dots j)$ with $i < j$ and $(i \dots ij)$ such that $i \rightarrow j \in E'$. By counting, these three groups each contribute with $d(d-1)/2$, $d(d-1)/2$ and $d-1$ rows, respectively, so in total this is exactly $d^2 -1 $ rows. We will denote this submatrix by $A$. 
        
    From now on when referring to rows of $A$ they will be $(ij)$, $(ij \dots j)$ or $(i \dots ij)$, which will implicitly inform whether it is a row coming from the 
    second- or $r$-th-order Lyapunov equation. The rest of the proof consists of showing that $A$ has full rank generically.

    As discussed in Section \ref{sec:proof-reorganize} we can prove that $A$ has full rank generically, that is, outside of a proper algebraic subset, by exhibiting a choice of parameters $(M, \mathcal{C}_2, \mathcal{C}_3)$ where $A$ has full rank, since the entries of $A$ can be written as rational functions in $M$, $\mathcal{C}_2$ and $\mathcal{C}_r$. 
    Let $\mathcal{C}_2$ and $\mathcal{C}_r$  be the identity matrix and identity $d \times \cdots \times d$ tensor, respectively and let $M \in \mathbb{R}^{E'}$ with diagonal $M_{ii} = -1/(r \zeta)$ and $M_{\beta \alpha} = 1$ for $\alpha \rightarrow \beta \in E'$ when $\alpha \neq \beta$. Since we picked $E'$ such that $G'$ (without the self-loops) is a polytree, this implies that we can use the polynomial expressions in $\zeta$ for the entries of $\Sigma$ and $\mathcal{K} $ given by Corollary \ref{cor::simplifiedTrekRuleDAGs}. 

   Thus, $\mathrm{det}(A)$ will be a polynomial in $\zeta$. The remaining idea of the proof is to show that this polynomial is not the zero polynomial. This is done by using Leibniz' formula for determinants to achieve an expression for the lowest degree non-vanishing term of $\mathrm{det}(A)$,
   \begin{equation}
   \label{eq::LeibnizFormula}
       \mathrm{det}(A) = \sum_{ \pi \in S_{d^2-1}} \mathrm{sgn}(\pi) \prod_{i = 0}^{d^2-1} A_{i \pi(i)}. 
   \end{equation}

    From Corollary \ref{cor::simplifiedTrekRuleDAGs} we see that a given non-zero entry of $A$ will have lowest degree in $\zeta$ equal to the total length of the shortest trek in $G'$ between the nodes indexing the entry of $\Sigma$ or $\mathcal{K} $ plus 1. Thus, only entries of $A$ equal to the diagonal entries $\Sigma_{ii}$ and $\mathcal{K} _{i \dots i}$ will have $\mathrm{low\ deg} ( \Sigma_{ii}) = \mathrm{low\ deg} ( \mathcal{K} _{i \dots i}) =  1 $, the lowest possible lowest degree. Similarly, the only way to obtain lowest degree equal to $2$ is the case where $i \neq j$ and $i \rightarrow j \in E'$, then $\mathrm{low\ deg} ( \Sigma_{ij}) = \mathrm{low\ deg} ( \mathcal{K} _{i \dots ij}) =  2 $, and $\mathrm{low\ deg} ( \mathcal{K} _{ij \dots j}) =  r $. For two distinct nodes with no edge between them, the lowest degree of $\Sigma_{ij}$ or $\mathcal{K} _{i \dots i j}$ will be at least three. 

  Equations \eqref{eq::A2SigmaEntries} and \eqref{eq::A3KEntries} yield that in row $(ij)$  there is a diagonal entry in the columns $(i \rightarrow j)$ and $(j \rightarrow i)$, and in row $(i \dots ij)$ there is a diagonal entry in column $(i \rightarrow j)$ and similarly for $(ij \dots j)$ it will only be in column $(j \rightarrow i)$. Thus, there will never be a diagonal entry in the self-loop columns in $A$ since we do not have the $(ii)$ or $(i \dots i)$ rows. 

    Now the claim is that it is possible to pick a permutation (there will be several) such that all non self-loop columns correspond to a diagonal entry and all self-loop columns will correspond to a $\Sigma_{ij}$ or $\mathcal{K} _{i \dots ij}$ with an edge $i \rightarrow j \in E'$. By the discussion of possible lowest degrees for entries of $\Sigma$ and $\mathcal{K} $ this is the term of lowest possible degree of the determinant of $A$ by Leibniz' formula \eqref{eq::LeibnizFormula}. 
     We now characterize precisely which permutations $\pi \in S_{d^2-1}$ obtain this degree. 

    We subdivide the rows and columns into the following four sets
    \begin{align*}
        \mathbf{R}_1 &=  \{ (ij) \; | \; i < j,\   i \rightarrow j \in E' \} \qquad   &\mathbf{C}_1 &= \{ (i \rightarrow j) \; | \; i \rightarrow j \in E' \} \\
        \mathbf{R}_2 &=  \{ (i \dots ij) \; | \; i < j,  \ i \rightarrow j \in E' \} \qquad &\mathbf{C}_2 &= \{(i \rightarrow i) \; | \; i \in [d] \setminus \{1 \} \}  \\ 
        \mathbf{R}_3 &=  \{ (ij) \; | \; i < j, \  i \rightarrow j \not \in E' \} \qquad  &\mathbf{C}_3 &= \{ (i \rightarrow j) \; | \; i < j, \ i \rightarrow j \not \in E' \}\\
        \mathbf{R}_4 &=  \{ (ij \dots j) \; | \; i < j \} \qquad  &\mathbf{C}_4 &= \{ (j \rightarrow i) \; | \; i < j \} .
    \end{align*}
    For the columns $\mathbf{C}_3$ the only row where it is possible to obtain a diagonal entry for $(i \rightarrow j) \in \mathbf{C}_3$ is exactly the matching row $ (ij) \in R_3$. Thus, any permutation $\pi \in S_{d^2-1}$ has to fix these rows to their corresponding columns. 

    Furthermore, for the rows $ (ij \dots j) \in \mathbf{R}_4$ the only possible way to obtain either $\mathcal{K} _{i \dots i}$ or $\mathcal{K} _{i \dots ij}$  is in columns $(j \rightarrow i)$ or $(i \rightarrow j)$, respectively. However, in column $i \rightarrow j$ it will always be possible to pick $\pi$ to obtain a diagonal entry from one of the other rows instead. Therefore, any permutation $\pi$ also has to fix the rows $\mathbf{R}_4$ to be mapped to the set of columns going against the topological order, $\mathbf{C}_4$. 

    By reordering the rows and columns (and only changing the determinant up to a sign) we can assume $\pi$ is just the identity permutation on  $\mathbf{R}_3 \cup \mathbf{R}_4$ to $\mathbf{C}_3 \cup \mathbf{C}_4$. By considering Leibniz' formula they will always contribute with a factor of 
    $$\left( \frac{r}{2} \right)^{d(d-1)/2 -(d-1)} \zeta^{d(d-1) -(d-1)}$$
    to the lowest degree term of the determinant. 

    We are left to consider the $2(d-1) \times 2(d-1)$ submatrix $A_{(\mathbf{R}_1 \cup \mathbf{R}_2), (\mathbf{C}_1 \cup \mathbf{C}_2)}$ with columns indexed by the edges in the polytree, $C_1$ and the self-loops except the one for the first source node, $\mathbf{C}_2$.
    Recalling the lexicographic ordering of the edges in $E'$, we will now order each of the sets of rows individually according to this as well and order the self-loops $\mathbf{C}_2$ \textit{not} according to the topological order of the nodes but according to when a node first appears in the lexicographically ordered edges $E' = \{ (e_{11} \rightarrow e_{12}, \dots, e_{(d-1) 1} \rightarrow e_{(d-1) 2} \}$, denote this $\{(i_2 \rightarrow i_2), \dots (i_d \rightarrow i_d) \}$. This is not guaranteed to coincide with the topological order, but it is possible, for example if $G'$ is a path. Thus, it is guaranteed that $i_2 = e_{12}$ and that $i_j$ will be at least $e_{(j-1) 1} $ or $e_{(j-1) 2}$. 
    So the matrix will be indexed like this:

{\footnotesize
\begin{center}
\begin{NiceTabular}[ cell-space-limits=4pt]{ c | c c c c c c}
  & $(e_{11} \rightarrow e_{12}) $ & $\cdots$ & $(e_{(d-1) 1} \rightarrow e_{(d-1) 2})$ & $(i_2 \rightarrow i_2)$ & $\cdots$ & $(i_{d} \rightarrow i_{d})$ \\
  \hline 
 $(e_{11} e_{12})$ & *  &  & & *  \\  
 $\vdots$ &  &    \\ 
 $(e_{(d-1)1} e_{(d-1)2})$ \\ 
 $(e_{11} \dots  e_{11} e_{12})$ \\ 
 $\vdots$ \\
 $(e_{(d-1)1} \dots e_{(d-1) 1} e_{(d-1)2})$
\end{NiceTabular}
\end{center}
}

To obtain the lowest possible degree, the first $d-1$ columns should all contribute with a diagonal entry in the product, and the last $d-1$ columns should contribute with a $\Sigma_{ij}$ or $ \mathcal{K} _{i \dots ij}$, potentially multiplied by $(r-1)$, with $i \rightarrow j \in E'$. For example, for the first row, the permutation has to map it to one of the columns marked with a '$*$' above . 

We number the rows and columns $1, \dots, 2(d-1)$ in the order explained above. For the first $d-1$ columns to correspond to diagonal entries, the permutation $\pi$ has to satisfy
\begin{align*}
    \pi(i) = i & \; \text{ if } \; i = 1, \dots,d-1 \qquad \text{ or } \qquad \pi(i) = i - (d-1) \; \text{ if } \; i = d, \dots, 2(d-1) ,
\end{align*}
for all $i = 1, \dots, 2(d-1)$.
In the first case we pick the possible diagonal entry in $A_2(\Sigma)$ and in the second case in $A_r(\mathcal{K} )$ and there are always exactly these two choices. They will contribute with $(r/2) \zeta$ and $\zeta$ to the lowest degree term, respectively. 

Then for the $(d-1)$ rows not picked out to have a diagonal entry, it has to be matched with a self-loop column instead. This is done consecutively, so first we consider either row $1$, $(e_{11} e_{12})$, or row $1 + (d-1)$, $(e_{11} \dots e_{11} e_{12})$, and whichever one was not picked for the first column  $e_{11} \rightarrow e_{12} $ has to be assigned column $e_{12} \rightarrow e_{12}$, since $e_{11} \rightarrow e_{11}$ was the column removed from the matrix, and so on. Depending on which $i$ was picked above the opposite happens now,
\begin{align*}
    \pi(i) = i +(d-1) & \text{ if } \; i = 1, \dots,d-1 \qquad \text{ or } \qquad \pi(i) = i  \text{ if }  \; i = d, \dots, 2(d-1),
\end{align*}
for all $i = 1..2(d-1)$. This will contribute with a factor $(r/2)^{2} \zeta^2$ or $\zeta^2$ (or $(r-1) \zeta^2$), respectively. 

If there is only one source in $G'$ each row with indices equal to those of edge $e_{i1} \rightarrow e_{i2} \in E'$, $(e_{i1} e_{i2})$ or $(e_{i1} \dots e_{i1} e_{i2})$ will be matched with the self-loop corresponding to the child, $e_{i2}$, when paired with a self-loop column
 because $e_{11} \rightarrow e_{11}$ is the only column not in the matrix. 

If there are multiple sources, then for all descendants of the first source $e_{11}$, the rows corresponding to the edge $e_{i1} \rightarrow e_{i2} \in E'$, $(e_{i1} e_{i2})$ or $(e_{i1} \dots e_{i1} e_{i2})$ will again always be matched with the child when matched with a self-loop column. However, for the nodes which are not descendants of the first source (so descendants of a different source and not also a descendant of the first source) the  rows corresponding to the edge $e_{i1} \rightarrow e_{i2} \in E'$, $(e_{i1} e_{i2})$ or $(e_{i1} \dots e_{i1} e_{i2})$ are matched with the self-loop of the parent $e_{i1}$, when matched up with a self-loop column. 

The paths starting at different sources in the graph than $e_{11}$ all have to converge with the path starting at $e_{11}$ because $G'$ is a tree. So the parent node is picked until you hit the spot where this path converges with the path which started at $e_{11}$. After the convergence point you then again always have to match a row with the child's self-loop as stated above. Such a convergence point must exist because $E'$ is a tree. However, the main point is that once it is decided which rows will be mapped to the first $d-1$ columns by $\pi$ everything else about $\pi$ is decided by going down the remaining rows in order. 

By the ordering of the rows and columns it further follows that $\pi$ will be a product of $2$-cycles. Since for each $i = 1, \dots, d-1$ either both 
\begin{equation*}
    \pi(i) = i \qquad \mathrm{ and } \qquad \pi(i + (d-1)) = i + (d-1)
\end{equation*}
or they result in a two cycle where the two are flipped
\begin{equation*}
    \pi(i) = i + (d-1) \qquad \mathrm{ and } \qquad \pi(i + (d-1)) = i. 
\end{equation*}

Thus, if $k$ is the number of the first $(d-1)$ first rows that are mapped to the first $(d-1)$ columns, then $\pi$ will be a product of $(d-1) - k$ two-cycles. Since we do not care about the sign of the determinant, the main point is that the permutations with even $k$ will always have the same sign and all the permutations with odd $k$ will always have the same sign, since a $2$-cycle has sign $-1$. 

One problem remains before the formula for the lowest degree term of the determinant of $A$, up to its sign, can be written. In the equation \eqref{eq::A3KEntries} for $A_r(\mathcal{K} )$  there can be a factor $r-1$ in front of $\mathcal{K} _{i \dots ij}$. This will happen exactly when in row $(e_{i1} \dots e_{i1} e_{i2})$ we pick the self-loop corresponding to the parent, i.e., $e_{i1} \rightarrow e_{i1}$. 
Let $q = 0, \dots, d-2$ denote the number of non-self-loop edges on the paths starting at a different source node than the first source node, $e_{11}$, until the convergence points with the path starting at $e_{11}$.
Then for each permutation $\pi$ such that we are in the first case above where $(ij)$ (so $A_2(\Sigma)$) contributes a diagonal entry and $i$ or $j$ is not a descendant of the first source node (so it is a descendant of another source before the convergence point) there will be a factor $r-1$ contributed by $A_r(\mathcal{K} )$ from the $(i \dots ij)$ row. So $q$ is the largest possible number of rows, which can contribute a $(r-1)$ factor. 


Thus, we can obtain the following expression for the lowest degree term using Leibniz rule by going through the permutations by how many indices $i$ where the first $(d-1)$ rows are mapped to the first $(d-1)$ columns. In other words, how many of the $A_2(\Sigma)$ rows are picked to contribute the diagonal entries. Thus, the sum over permutations can be written as a sum over $i$:



\begin{align*}
    \mathrm{low\ deg\ term}( \mathrm{det}(A_{(\mathbf{R}_1 \cup \mathbf{R}_2), (\mathbf{C}_1 \cup \mathbf{C}_2)})) = \zeta^{2(d-1)} \zeta^{(d-1)}  \sum_{i = 0}^{d-1}  \left(\frac{r}{2} \right)^{i}  \left( \left(\frac{r}{2} \right)^2 \right)^{d-1-i} (-1)^{i} c(d,q,i),
\end{align*}

where $$c(d,q,i) = \sum_{j = 0}^{i} (r-1)^{j} \binom{q}{j} \binom{d-1-q}{i-j}.$$ For a given number $i$ of $A_2(\Sigma)$-rows picked for diagonal entries, the sum $c(d,q,i)$ counts the added contribution by potentially having to pick a factor $(r-1)$, depending on the graph, in the $r$-th-order rows. The summation index $j$ counts the number of columns picked where an $(r-1)$ is included, and 
the binomial coefficient product counts the number of ways, potentially zero, that the columns can be picked such that exactly $j$ columns with a factor $(r-1)$ are picked. 
By Lemma \ref{lem::forsum}
\begin{align*}
    \mathrm{low\ deg\ term}( \mathrm{det}(A_{(\mathbf{R}_1 \cup \mathbf{R}_2), (\mathbf{C}_1 \cup \mathbf{C}_2)})) &= \zeta^{2(d-1)} \zeta^{(d-1)} \left(\frac{r}{2} \right)^{d-1}   \sum_{i = 0}^{d-1} \left(\frac{r}{2} \right)^{d-1-i} (-1)^{i} c(d,q,i) \\ 
    &=  \zeta^{2(d-1)} \zeta^{(d-1)} \left(\frac{r}{2} \right)^{d-1} \left(\frac{r}{2} -1 \right)^{d-1} (-1)^{q}. 
\end{align*}

Thus, the coefficient in front of the smallest degree term in the determinant of all of $A$ is not zero, it is 
$$\left(\frac{r}{2} \right)^{d-1} \left(\frac{r}{2} -1 \right)^{d-1} (-1)^{q} \left( \frac{r}{2} \right)^{d(d-1)/2 -(d-1)}.$$ 

The determinant of $A$ is therefore not the zero polynomial in $\zeta$, and there will exist a $\zeta$-value such that the corresponding $M$, $\mathcal{C}_2$ and $\mathcal{C}_r$ result in the determinant of $A$ being non-zero.
\end{proof}

\begin{remark}   
In the proof we show that the rank of $A$ is $d^2-1$ by removing the column corresponding to the self-loop of the first node in the topological ordering and showing this matrix has rank $d^2-1$. However, the proof will work by removing any of the self-loop columns. The only thing this will change is how to determine $q$, which counts the columns where it is possible to obtain a factor $(r-1)$. 
\end{remark}

\begin{remark}
From the matrix $A$ we could also remove columns corresponding to the zero entries of $M$ and only solve for the non-zero entries. This would not change the proof in any substantial way. It would only make the set of rows $R_3$ and $R_4$, and corresponding column sets $\mathbf{C_3}$ and $\mathbf{C_4}$, smaller, and their fixed contribution to the determinant later in the proof would be smaller as well. But no arguments change or depend on how many of the rows or columns in these sets are included as long as the only deleted rows 
and columns correspond to zero entries in $M$ given by the graph.  
\end{remark}

\begin{proof}[ of Corollary \ref{cor::RankConnectedComponents}]
     
     Let $S_1, \dots, S_m$ be the connected components of $G$. Let $\mathbf{C}_{S_i}$ and $\mathbf{R}_{S_i}$ denote all columns and rows, respectively where all indices  belong to the connected component $S_i$. Let $\mathbf{C}_{\mathrm{mix}}$ and $\mathbf{R}_{\mathrm{mix}}$ denote the rows and columns with indices in more than one component. 
    By definition of $A_2$ and $A_r$, see equations \eqref{eq::A2SigmaEntries} and \eqref{eq::generalAmatrixentry}, we see that $A_{\mathrm{off}}$ exhibits the following block structure if we organize the rows and columns by connected components
    \begin{equation*}
        A_{\mathrm{off}} = \begin{pNiceArray}{cccc}[first-row,first-col, nullify-dots]
       & \mathbf{C}_{S_1}    & \cdots &  \mathbf{C}_{S_m}       & \mathbf{C}_{\mathrm{mix}}  \\
\mathbf{R}_{S_1} \; \; & A_{\mathrm{off}}^{S_1} & 0 & \cdots & 0\\
\vdots \; \; \; \;    & 0 & \ddots &  & \vdots  \\
\mathbf{R}_{S_m} \; \;     & \vdots &  & A_{\mathrm{off}}^{S_m} & 0  \\
\mathbf{R}_{\mathrm{mix}}      & 0 & \cdots & 0 & A_{\mathrm{off}}^{\mathrm{mix}}  \\
\end{pNiceArray}. 
    \end{equation*}
The zero-pattern is explained as follows: if we consider $\alpha \rightarrow \beta \in S_i$ then $A_2(\Sigma)_{(kj), \alpha \rightarrow \beta} = 0$ if $k \not \in S_i$ or $j \not \in S_i$. Similarly, for the $r$-th-order Lyapunov equation rows. Explicitly, by equation \eqref{eq::generalAmatrixentry} we have that $A_r(\mathcal{K} )_{i_1, \dots, i_r, \alpha \rightarrow \beta}$ is given as some integer multiple of $\mathcal{K} _{\alpha i_2, \dots, i_r}$ given that $\beta = i_1$, and similarly if $\beta$ was equal to one of the other indices. Thus, an entry can only be non-zero if $\beta$ overlaps with $(i_1, \dots i_r)$, and there is a trek between $\alpha$ and the remaining $i$'s, so $\alpha$ and the remaining $i$'s need to belong to the same connected component.

By the block diagonal structure 
\begin{equation}
\label{eq::AndRankConnectedComp}
    \mathrm{rank} \left(A_{\mathrm{off}} \right) = \sum_{i = 1}^m\mathrm{rank} \left(A^{S_i}_{\mathrm{off}} \right) + \mathrm{rank} \left(A^{\mathrm{mix}}_{\mathrm{off}} \right).
\end{equation}
Each block, $A^{S_i}_{\mathrm{off}}$ can have at most rank  $|S_i| -1$ showing the first part of the Corollary, $\mathrm{rank} \left(A_{\mathrm{off}} \right) \leq d-m.$

By Theorem \ref{thm::GenericIdResult}, $\mathrm{rank}(A^{S_i}_{\mathrm{off}}) = |S_i| -1$ generically. This would be enough for the m-dimensional identifiability of the graph with $m$ connected components. However, it is not enough in terms of understanding the exceptional set for the connected graphs. 

To show that $A^{\mathrm{mix}}_{\mathrm{off}}$ has full rank we notice that the square submatrix only indexed by the $r$-th-order off-diagonal rows of the form $(i \dots ij)$ and all columns in $\mathbf{C}_{\mathrm{mix}}$ has full rank. This is seen by noting that this matrix exhibits a block diagonal structure. For each connected component $S_i$ and each $s \in S_i$ group the columns $t \rightarrow s$ and rows $(t \dots t s)$ for $t \in [d] \setminus S_i$. The argument for the zero-pattern is completely analogous to that for $A_{\mathrm{off}}$ by either using no overlap in the sink and indices of the rows or when there is a matching a lack of trek between the source and the remaining indices. 

Let $s \in S_i$ for a given connected component and number the elements in $t_1, \dots t_l \in [d] \setminus S_i$, then the block on the diagonal corresponding to $s$ has the form 
\begin{equation*}
    (D_s)_{(t_i \dots t_i s), (t_j \rightarrow s)} = \mathcal{K} _{t_i \dots t_i t_j}. 
\end{equation*}
Thus, if we pick $M$ to be diagonal, each $D_s$ will become diagonal with $\mathcal{K} _{t_1 \dots t_1}, \dots, \mathcal{K} _{t_l \dots t_l}$ showing that each $D_s$ generically has full rank. Thus, the entire square matrix generically has full rank, as does $A^{\mathrm{mix}}_{\mathrm{off}}$. 
Combing this we obtain by equation \eqref{eq::AndRankConnectedComp}
\begin{equation*}
     \mathrm{rank}(A_{\mathrm{off}}) = d^2 - m,
\end{equation*}
so the rank drops by exactly the number of connected components generically. 
\end{proof}

\appendixtwo

\section{Appendix -- additional proofs and supplementary material}

\subsection{Proofs about the estimator} \label{sec::Proofsabouttheest}

\begin{proof}[ of Theorem \ref{thm::AsympoticsSingularValueEst}]
By the law of large numbers, the empirical moments up to order $k$ converge in probability to the true moments.
Since the cumulants can be obtained by a continuous, and smooth, transformation of the moments, see, e.g., Section 2.4.3 by \cite{mccullagh2018tensor} for an explicit formula we can apply the continuous mapping theorem, Theorem 2.3 by \cite{Vaart_1998}, to obtain that the empirical cumulants also converge in probability to the true cumulants. As the matrix $A$ is assumed to have entries equal to cumulants, up to scaling by a real number, it follows by another use of the continuous mapping theorem that
$$\hat{A}_n \xrightarrow{P} A \text{ for } n \rightarrow \infty.$$

By definition, $\nu_{\mathrm{min}}(\hat{M})$ can be found as the unit eigenvector of $\hat{A}^T \hat{A}$ corresponding to the smallest eigenvalue of $\hat{A}^T \hat{A}$ picked such that $\mathrm{tr}(\hat{M}) < 0$. Let $G$ denote this function. Then $G$ is continuous and smooth in an open neighborhood of the true $A$. This holds since $G$ is the composition of two maps, which are smooth in a neighborhood of $A$. 
Clearly, the map $X \mapsto X^T X$ is smooth. 
The unit eigenvector map for a simple eigenvalue of a real symmetric matrix, $X_0$, is also smooth
in a neighborhood of $X_0$, see Theorem~1 by \cite{Magnus_1985}. To apply this result for $A^TA$ we note that we have assumed $\mathrm{ker}(A) = \mathrm{span}(\mathrm{vec}(M))$, so the 
smallest eigenvalue of $A^TA$ is $0$, and since the kernel is one-dimensional it is a simple eigenvalue. 
Thus, 
we can apply the continuous mapping theorem to conclude that $$\hat{M}_n \xrightarrow{P} M \text{ for } n \rightarrow \infty.$$

If the distribution, additionally, has finite $2k$-th-order moment, asymptotic normality follow by an application of the delta method. By the central limit theorem, the empirical moments up to order $k$ will be asymptotically normally distributed. Thus, since the cumulants can be obtained by a differentiable function of the moments it follows by the delta method that the empirical cumulants are also asymptotically normally distributed, as written in equation \eqref{eq::AsymptoticNormalityCumulants}. As $A = \mathcal{A} \kappa$ for $\mathcal{A}$ a fixed matrix, we obtain
\begin{align*}
    \sqrt{n}(\mathrm{vec}(\hat{A}_n - A)) & \xrightarrow{D} \mathcal{N} \left(0, \mathcal{A} \Omega \mathcal{A}^T \right)
\end{align*}
for $n \rightarrow \infty$ by use of the delta method. 

As argued above $M$ and $\hat{M}$ are obtained by applying the function $G$ to $A$ and $\hat{A}$, 
respectively, and since $G$ is smooth in a neighborhood of the true $A$, it follows by the delta method that 
$$
\sqrt{n}(\mathrm{vec}(\hat{M}_n - M))  \xrightarrow{D} \mathcal{N} \left(0,  D G (A)  \mathcal{A}  \Omega  \mathcal{A}^T (DG(A))^T  \right)
$$
for $n \rightarrow \infty$, where $DG(A)$ denotes the Jacobian evaluated in $A$. By Lemma \ref{lem::JacobianOfSingularVector}, we have 
$$
    DG(A) = - \mathrm{vec}(M)^T \otimes A^{+},
$$
and we obtain the asymptotic covariance matrix specified in the theorem. 
\end{proof}

We derive the Jacobian of a vectorized version of a map of matrices. This is done using notation and conventions from Chapter 6 by \cite{magnus_neudecker_2019}.

\begin{lemma} \label{lem::JacobianOfSingularVector}
    Let $G: \mathbb{R}^{b \times d^2} \rightarrow \mathbb{R}^{d^2}$ denote the function mapping a matrix $X$ to $v_{\min}(X)$, the right unit singular vector of the smallest singular value of $X$. Consider a matrix $A \in \mathbb{R}^{b \times d^2}$ whose smallest singular value is simple and equal to zero with right singular vector denoted $\mathrm{vec}(M)$. Then in a neighborhood of $A$, $G$ is a smooth function with Jacobian evaluated in $A$ given by 
    \begin{equation}
        DG(A) = \left. \frac{\partial \mathrm{vec}(G)}{\partial \mathrm{vec}(X)^T} \right |_{X = A}  = - \mathrm{vec}(M)^T \otimes A^{+},
    \end{equation}
    where $A^{+}$ denotes the Moore-Penrose inverse. 
\end{lemma}
\begin{proof}
    We first note that the smallest singular value and its corresponding right singular vector of a matrix $X$ can be found as the corresponding smallest eigenvalue and eigenvector of $X^TX$. Thus, $G$ can be viewed as the composition of maps $\mathrm{eig}_{\mathrm{min}} \circ F$, where $F$ maps a matrix $X$ to $X^T X$ and $\mathrm{eig}_{\mathrm{min}}$ maps a matrix to the unit 
    eigenvector of its smallest eigenvalue. We compute the Jacobian of each of these maps.
    Theorem 1 by \cite{Magnus_1985} gives that the Jacobian of $\mathrm{eig}_{\mathrm{min}}$ 
    at $A^TA$ is given as 
    $$D\mathrm{eig}_{\mathrm{min}}(A^TA) = \left. \frac{\partial \mathrm{eig}_{\mathrm{min}}}{\partial(\mathrm{vec}(X))^T} \right|_{X = A^TA} = -\mathrm{vec}(M)^T \otimes (A^TA)^{+},$$ 
    since the smallest eigenvalue is zero with eigenvector $\mathrm{vec}(M)$.
    
    By letting $\varepsilon > 0$ and $Y$ any matrix of the same dimension as $X$ we can write
    \begin{align*}
        \mathrm{vec}(F(X + \varepsilon Y)) &= \mathrm{vec} \left((X + \varepsilon Y)^T(X + \varepsilon Y)\right) \\
        &=\mathrm{vec}(X^TX) + \varepsilon \mathrm{vec}\left(Y^T X + X^T Y \right)  + \varepsilon^2 \mathrm{vec}( Y^T Y ) \\
        &= \mathrm{vec}(X^TX) + \varepsilon ((X^T \otimes I)K + I \otimes X^T) \mathrm{vec}(Y)  + \varepsilon^2 \mathrm{vec}( Y^T Y ),
    \end{align*}
    by using rules of vectorization and Kronecker products, see Chapter 2 by \cite{magnus_neudecker_2019}, and where $K$ denotes the commutation matrix that switches the vectorization of a matrix with the vectorization of its transpose. From this we can conclude that 
    $$DF(A) = \left. \frac{\partial F}{\partial(\mathrm{vec}(X))^T} \right|_{X = A} = (A^T \otimes I)K + I \otimes A^T.$$ 
    Thus, by the chain rule we obtain the following expression for the Jacobian of $G$ at $A$
    \begin{align*}
        DG(A) = \left. \frac{\partial \mathrm{vec}(G)}{\partial \mathrm{vec}(X)^T} \right|_{X = A}
        &= D\mathrm{eig}_{\mathrm{min}}(A^TA) DF(A) \\
        &= -\mathrm{vec}(M)^T \otimes (A^TA)^{+} ((A^T \otimes I)K + I \otimes A^T) \\ 
        &= -(\mathrm{vec}(M)^T A^T \otimes (A^TA)^{+})K -\mathrm{vec}(M)^T \otimes (A^TA)^{+} A^T \\
        &= -\mathrm{vec}(M)^T \otimes (A^TA)^{+} A^T = -\mathrm{vec}(M)^T \otimes A^{+},
    \end{align*}
    by using that $A \mathrm{vec}(M) = 0$. 
\end{proof}

\begin{proof}[of Corollary \ref{cor::RewriteAsymptoticCov}]
We rewrite
$$\left(\mathrm{vec}(M)^T \otimes A^{+} \right) \mathcal{A}$$
in a way to avoid forming the large matrices $\mathrm{vec}(M)^T \otimes A^{+}$ and $\mathcal{A}$. The columns of $\mathcal{A}$ are indexed by the unique vectorizations of the second- and $r$-th-order cumulants $\Sigma_{11}, \dots, \Sigma_{dd}, \mathcal{K}_{1 \cdots 1}, \dots, \mathcal{K}_{d \cdots d}$. We let the second- and r-th-order cumulants jointly be given by $\kappa$ and let $j$ index this set. If we consider column $j$ of $\mathcal{A}$, $\mathcal{A}_{*,j}$ as being written as the vectorization of a $D \times d^2$ matrix. Then this matrix 
can be seen to be $A$ evaluated at the point where 
every cumulant except $j$ is equal to zero and $j = 1$. If we let $\kappa_j = 1$ denote the vector of second- and $r$-th-order cumulants where the $j$-th is set to one and the remaining to zero, we obtain that 
$$\mathcal{A}_{*,j} = \mathrm{vec} (A(\kappa_j = 1)).$$

Then 
\begin{equation*}
    \left(\mathrm{vec}(M)^T \otimes A^{+} \right) \mathcal{A}_{*,j} = A^{+} A(\kappa_j = 1) \text{vec}(M),
\end{equation*}
and therefore, 
\begin{equation*}
    \left(\mathrm{vec}(M)^T \otimes A^{+} \right) \mathcal{A}  = A^{+} \left[ A(\kappa_1 = 1) \mathrm{vec}(M) \cdots A(\kappa_q = 1) \mathrm{vec}(M) \right]. 
\end{equation*}
Each column $A(\kappa_j = 1) \mathrm{vec}(M)$ is by definition of $A$ the off-diagonal Lyapunov equation, of order $2$ and $r$ combined, evaluated at $M$, and the cumulant $\kappa_j$ for which the $j$-th-order cumulant is one and the rest are zero. Thus, this matrix is just equal to the block diagonal matrix
\begin{equation*}
    \left[ A(\kappa_1 = 1) \mathrm{vec}(M) \cdots A(\kappa_q = 1) \mathrm{vec}(M) \right] = B_{2,r}(M) = \begin{pmatrix}
        B_2(M) & 0 \\
        0 & B_r(M) \\
    \end{pmatrix},
\end{equation*}
with $B_k(M)$ defined by equation \eqref{eq::Br(M)vec}, so we obtain that 
\begin{equation*}
\left(\mathrm{vec}(M)^T \otimes A^{+} \right) \mathcal{A} = A^{+} B_{2,r}(M),
\end{equation*}
which finishes the proof. 
\end{proof}

\subsection{General Lyapunov models and trek formulas}
\label{Appendix::treks} 
For our main results we assume that the continuous Lyapunov model has diagonal cumulants 
of the driving Lévy process. To encode a more general sparsity pattern of the cumulant we 
introduce graphs with blunt edges, as was done in covariance case by \cite{varando20a}. In 
this appendix we will provide this definition for any order cumulant of a continuous 
Lyapunov model and derive general trek formulas in this setup. 

\begin{definition}
\label{def::Compatiblewithgraph}
Let $G = ([d], E,B)$ be a $k$-th-order mixed graph with $E$ denoting the set of directed edges ($\rightarrow$) and $B$  denoting the $k$-fold blunt edges ($\unEdge$). Let 
\begin{align*}
     \mathbb{R}^{E} &= \{ M \in \mathbb{R}^{d \times d} \; | \; M_{ij} = 0 \; \mathrm{ if } \; j \rightarrow i \not \in E \} \\
     \mathrm{Sym}^k(\mathbb{R}^d)^B &= \{ C \in \mathrm{Sym}^k(\mathbb{R}^d) \; | \; C_{i_1, \dots, i_k} = 0 \; \mathrm{ if } \; (i, \dots, i_k) \not \in B \}.
\end{align*}
If $(M, \mathcal{C}_k) \in \mathbb{R}^{d \times d} \times \mathrm{Sym}^k(\mathbb{R}^d)$ satisfy that $M \in \mathbb{R}^{E}$ and $\mathcal{C}_k \in \mathrm{Sym}^k(\mathbb{R}^d)^B$ we will say that they are \textit{compatible} with the mixed graph $G$. 
\end{definition}

In Figure \ref{fig::GeneralModel} there is an example of a second-order mixed graph and corresponding zero-pattern of $(M, \mathcal{C}_2)$. The directed edges are directed and drawn in black and the blunt edges are drawn in blue and undirected, so the order the nodes are written in is irrelevant. 
    \begin{figure}[H]
    \centering
\begin{center}
    \begin{tikzpicture}
      \node[circle, draw, minimum size=0.5cm] (1) at (-1.5,0) {1};
      \node[circle, draw, minimum size=0.5cm] (2) at (0,1) {2};
      \node[circle, draw, minimum size=0.5cm] (3) at (0,-1) {3};
      \node[circle, draw, minimum size=0.5cm] (4) at (1.5,0) {4};
      
      \draw[->,  >=stealth] (1) edge[loop above] (1);
      \draw[->,  >=stealth] (2) edge[loop above] (2);
      \draw[->,  >=stealth] (3) edge[loop below] (3);
      \draw[->,  >=stealth] (4) edge[loop above] (4);
      \draw[->,  >=stealth] (1) edge (3);
      \draw[->,  >=stealth] (4) edge (2);
      \draw[->,  >=stealth] (3) edge (4);
      \draw[->,  >=stealth] (2) edge [bend right] (3);
      \draw[->,  >=stealth] (3) edge[bend right] (2);

      \draw[edgeblunt, shorten <=1pt, shorten >=1pt, loop left, blue] (1) to (1);
      \draw[edgeblunt, shorten <=1pt, shorten >=1pt, loop left, blue] (2) to (2);
      \draw[edgeblunt, shorten <=1pt, shorten >=1pt, loop left, blue] (3) to (3);
      \draw[edgeblunt, shorten <=1pt, shorten >=1pt, loop right, blue] (4) to (4);
      \draw[edgeblunt, shorten <=1pt, shorten >=1pt, blue] (1) edge (2);
      \draw[edgeblunt, shorten <=1pt, shorten >=1pt, bend right, blue] (3) to (4);
      \draw[edgeblunt, shorten <=1pt, shorten >=1pt, bend left = 60, blue] (3) to (1);
      \draw[edgeblunt, shorten <=1pt, shorten >=1pt, bend left = 35, blue] (2) to (4);

    \node at (5, 0) {$
    M = \left( 
    \begin{array}{cccc}
     \;  * \; \; \; & 0 \; \; \; & 0 \; \; \; & 0 \; \;  \\
     \; 0 \; \; \; & * \; \; \; & * \; \; \; & * \; \;  \\
     \;  * \; \; \; & * \; \; \; & * \; \; \; & 0 \; \; \\
    \;  0 \; \; \; & 0 \; \; \; & * \; \; \; & * \; \; \\
    \end{array}
    \right) 
    $};

     \node at (9, 0) {$
    \mathcal{C}_2 = \left( 
    \begin{array}{cccc}
     \;  * \; \; \; & * \; \; \; & * \; \; \; & 0 \; \;  \\
     \; * \; \; \; & * \; \; \; & 0 \; \; \; & * \; \;  \\
     \;  * \; \; \; & 0 \; \; \; & * \; \; \; & * \; \; \\
    \;  0 \; \; \; & * \; \; \; & * \; \; \; & * \; \; \\
    \end{array}
    \right) 
    $};
      
    \end{tikzpicture} 
    \end{center}
    \caption{Example of a second-order mixed graph $G = ([4], E,B)$ and corresponding zero-pattern of pair $(M, \mathcal{C}_2)$ compatible with $G$.}  
    \label{fig::GeneralModel}
\end{figure}
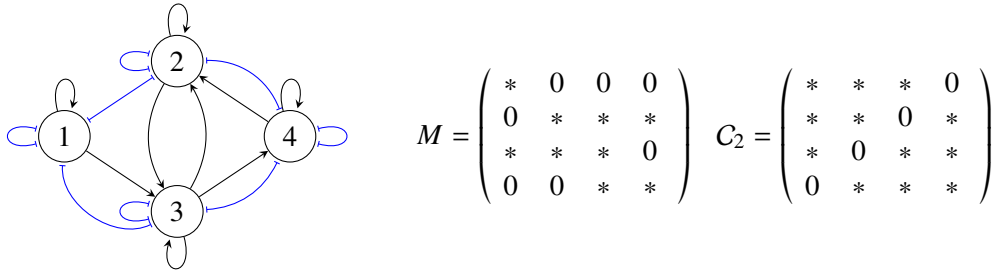

Since the steady-state solution to equation \eqref{eq::SDErep} only exists if the drift matrix $M$ is stable, we let $\mathbb{R}^{E}_{\mathrm{stab}}$ denote the subset of $\mathbb{R}^{E}$ that consists of stable matrices as in the main paper. Furthermore, some cumulants, e.g., the covariance, satisfy positive definiteness, so when such further positivity constraints are applicable, this should also be included in the model definitions. 
As in the main paper, we will always assume that the directed part of the graph $G = ([d], E)$ contains all self-loops. 

We can define the continuous Lyapunov model of a certain order $k$ by the set of distributions in $\mathcal{P}_G$, where $G$ represents only the directed part of the mixed graph, and where the Lévy process has finite $k$-th-order moment and the $k$-th-order cumulant of the Lévy process is compatible with the blunt edges in $G$, i.e., $\mathcal{C}_k \in \mathrm{Sym}^k(\mathbb{R}^d)^B$. As in the main paper, we will also consider the finite-dimensional set of cumulants of order $k$, which can arise in this way, as solutions to the $k$-th-order continuous Lyapunov equation with $(M, \mathcal{C}_k)$ corresponding to a mixed graph. 
\begin{definition}
Let $G = ([d], E, B)$ be a $k$-th-order mixed graph on $d$ nodes. Then the continuous $k$-th-order Lyapunov cumulant model is the set of $k$-th-order tensors in the set
\begin{equation*}
    \mathcal{M}^{k}_{G} = \{ \mathcal{K} \in \mathrm{Sym}^k(\mathbb{R}^d) \; | \; \mathcal{K}  \times_1 M + \cdots \mathcal{K}  \times_k M + \mathcal{C}_k = 0 \text { for } M \in \mathbb{R}^{E}_{\mathrm{stab}} \text{ and } \mathcal{C}_k \in \mathrm{Sym}^k(\mathbb{R}^d)^B \}. 
\end{equation*}
\end{definition}
In the following we will derive two different general trek rules for the $k$-th-order Lyapunov equation.

When considering the integral equation \eqref{eq::generalSol} for the  solution to the $k$-th-order Lyapunov equation there are at least two different options of how to derive a trek formula from it. It depends upon whether you are willing to make an assumption on the spectral radius of $M$ or not. The first option is to consider
\begin{equation} \label{eq::K(s)eq}
    \mathcal{K} (s) = \int_{0}^{s} \mathcal{C}_k \times_1 e^{M t} \times_2 e^{M t} \cdots \times_{k} e^{M t} \mathrm{d} t,
\end{equation}
which clearly satisfies $\mathrm{lim}_{s \rightarrow \infty} \mathcal{K} (s) = \mathcal{K} $, and then obtain  a trek formula expression for $\mathcal{K} (s)$ as is done in the 
case $k=2$ by \cite{varando20a}.
The other option requires an assumption which allows one to exchange integration and sums, as is done for $k=2$ by \cite{hansen2024trekrulelyapunovequation}. Before we proceed with either, we will first give a general definition of a $k$-trek. 
\begin{definition}
    Let $G = ([d],E,B)$ be a $k$-th-order mixed graph. A $k$-trek $\tau$ between the nodes $i_1, \dots, i_k$ consists
    of the following: a $k$-fold blunt edge $((j_1)_1, \dots ,(j_k)_1)$ in the graph $G$ and $k$ directed walks in $G$
    \begin{align*}
        & \pi_1 :  (j_1)_1 \rightarrow (j_1)_2 \cdots \rightarrow i_1 \\ 
        &\vdots \\ 
        & \pi_k: (j_n)_1 \rightarrow (j_n)_2 \cdots \rightarrow i_k. 
    \end{align*}
    Let $\ell_{i_m}(\tau)$ denote the length of the walk going to $i_m$ and $\mathcal{T}(i_1, \dots, i_k)$ denote the set-up of all treks between the nodes $i_1, \dots, i_k$. 
\end{definition}
Note that each of the walks is allowed to have length $0$. If the blunt part of the graph only contains the self-loop blunt edges (corresponding to the $\mathcal{C}$ being diagonal) we can also  define the trek from only a directed graph, because in this case the top nodes have to be the same in all the paths constituting the trek. 

\begin{example}
    We can compare the 2-treks in the graph in Figure \ref{fig::ZeroPatternEx} and Figure \ref{fig::GeneralModel}, where the directed part of the graph is the same, but the blunt edges are different. In Figure \ref{fig::ZeroPatternEx} there are not drawn, which corresponds there only being the self-loop blunt edges and $\mathcal{C}_2$ diagonal. We only consider the simple 2-treks meaning the ones without directed loops, since otherwise there are infinitely many but they all have to arise from adding loops to one of the simple treks. We provide two examples of pairs of nodes and the treks between them. If the top of the trek is a self-loop blunt edge $v \unEdge v$ we do by convention not write the blunt edge but just the given node in the middle of the trek as $\leftarrow v \rightarrow$. 

    \textbf{Treks between 1 and 2:} In Figure \ref{fig::ZeroPatternEx} there are none, in Figure \ref{fig::GeneralModel} there are two simple treks
    \begin{align*}
        1 \unEdge 2, \ 1 \unEdge 3 \rightarrow 2. 
    \end{align*}

\textbf{Treks between 2 and 4:} In Figure \ref{fig::ZeroPatternEx} we have
\begin{align*}
    2 \leftarrow 4, \ 2 \rightarrow 3 \rightarrow 4, \ 2 \leftarrow 3 \rightarrow 4, 
    \ 2 \leftarrow 3 \leftarrow 1 \rightarrow 3 \rightarrow 4.
\end{align*}
In Figure \ref{fig::GeneralModel} there are all the same as in Figure \ref{fig::ZeroPatternEx} as well as all the following
\begin{align*}
    &2 \unEdge 4, \\
    &2 \unEdge 1 \rightarrow 3 \rightarrow 4, 2 \leftarrow 3 \leftarrow 1 \unEdge 2 \rightarrow 3 \rightarrow 4,  \\ 
    &2 \leftarrow 3 \unEdge 1 \rightarrow 3 \rightarrow 4, 2 \leftarrow 4 \leftarrow 3 \leftarrow 1 \unEdge 3 \rightarrow 4, 2 \leftarrow 3 \unEdge 1 \rightarrow 3 \rightarrow 4, 2 \leftarrow 4 \leftarrow 3 \unEdge 1 \rightarrow 3 \rightarrow 4 \\
    &2 \leftarrow 3 \unEdge 4, 2 \leftarrow 4 \leftarrow 3 \unEdge 4, 2 \leftarrow 4 \unEdge 3 \rightarrow 4. 
\end{align*}

\end{example}

In order to prove the trek formulas, we provide the definition of a trek polynomial. 

\begin{definition}
Let $G = ([d],E,B)$ be a $k$-th-order mixed graph and let $(M,\mathcal{C})$ be a pair of a $d\times d$ stable matrix and $k$-th-order tensor that is compatible with $G$ then for $\tau$ a $k$-trek with the blunt edge 
$((j_1)_1, \dots, (j_n)_1)$ as its center, we define 
\begin{align*}
    \omega(M,C,\tau) = \mathcal{C}_{(j_1)_1, \dots, (j_k)_1} \prod_{m = 1}^k \prod_{\alpha \rightarrow \beta \in \pi_m} M_{\beta \alpha}.  
\end{align*}
This is called the trek polynomial corresponding to the trek $\tau$. 
\end{definition}    

We will prove two different trek rules. In both cases, we write the proof for $k = 3$ for notational convenience, it is easy to see from the 
proof how it generalizes to general $k$ and we will write this result as well. 

\begin{proposition}
\label{prop::ThirdOrderTrekwithS}
    Let $G = ([d],E,B)$ be a third-order mixed graph and let $(M,\mathcal{C})$ be a pair of a $d\times d$ stable matrix and $\mathcal{C}$ third-order tensor that is compatible with $G$ then for $\mathcal{K}(s)$ as defined by equation \eqref{eq::K(s)eq} (for $k = 3$) we obtain 
    \begin{align*}
        (\mathcal{K} (s))_{i_1 i_2 i_3} =  \sum_{\tau \in \mathcal{T}(i_1, i_2, i_3)} 
    \frac{s^{\ell_{i_1}(\tau)+\ell_{i_2}(\tau)+\ell_{i_3}(\tau)+1}}{\ell_{i_1}(\tau)! \ell_{i_2}(\tau)! \ell_{i_3}(\tau)! (\ell_{i_1}(\tau)+\ell_{i_2}(\tau)+\ell_{i_3}(\tau)+1)}   
    \omega(M,\mathcal{C},\tau).
    \end{align*}
\end{proposition}
\begin{proof}
    First we consider the general expression for $\mathcal{K}(s)$. We obtain the following by using the series expansion for the matrix exponential
    and using that we can exchange the sum and the integral since $M$ is stable so the infinite sum is convergent. 
    \begin{align*}
        \mathcal{K} (s) &= \int_{0}^{s} \mathcal{C} \times_1 e^{M t} \times_2 e^{M t} \times_3  e^{M t} \mathrm{d} t \\ 
               &=  \int_{0}^{s} \mathcal{C} \times_1 \sum_{n = 0}^{\infty} \frac{t^n}{n!} M^n \times_2 
               \sum_{m = 0}^{\infty} \frac{t^m}{m!} M^m \times_3  \sum_{k = 0}^{\infty} \frac{t^k}{k!} M^k \mathrm{d} t \\ 
               &= \sum_{n = 0}^{\infty} \sum_{m = 0}^{\infty} \sum_{k = 0}^{\infty} \int_{0}^{s} \mathcal{C} \times_1 \frac{t^n}{n!} M^n 
               \times_2 \frac{t^m}{m!} M^m \times_3 \frac{t^k}{k!} M^k \mathrm{d} t \\ 
               &= \sum_{n = 0}^{\infty} \sum_{m = 0}^{\infty} \sum_{k = 0}^{\infty} \int_{0}^{s} \frac{t^{n+m+k}}{n! m! k!}  
               \mathcal{C} \times_1  M^n \times_2 M^m \times_3  M^k \mathrm{d} t \\ 
               &= \sum_{n = 0}^{\infty} \sum_{m = 0}^{\infty} \sum_{k = 0}^{\infty}  \frac{s^{n+m+k+1}}{n! m! k! (n+m+k+1)} 
               \mathcal{C} \times_1  M^n \times_2 M^m \times_3  M^k
    \end{align*}
    To derive an expression for the $(i_1 i_2  i_3)$-th entry we need an expression for the $(i_1 i_2  i_3)$-th  entry of $\mathcal{C}\times_1~M^n\times_2~M^m\times_3~M^k$. By definition
\begin{align*}
    (\mathcal{C} \times_1  M^n \times_2 M^m \times_3  M^k)_{i_1 i_2 i_3 } &= \sum_{\alpha = 1}^{d} 
    \sum_{\beta = 1}^{d} \sum_{\gamma = 1}^{d} 
    \mathcal{C}_{\alpha \beta \gamma} \cdot (M^{n})_{i_1 \alpha} \cdot (M^{m})_{i_2 \beta} \cdot (M^{k})_{i_3 \gamma}. 
\end{align*} 
Since $M$ is an adjacency matrix $(M^{n})_{i_1 \alpha}$ is the sum of all walk polynomials of directed walks in $G$ from $\alpha$ 
to $i_1$ of length $n$. Thus, the sum above will be the sum of the trek polynomials over all treks between $i_1, i_2$ and $i_3$ where the three walks going to 
$i_1$, $i_2$ and $i_3$ have lengths $n$, $m$ and $k$, respectively. Combining this with the expression above where we sum over all 
$n,m$ and $k$ we obtain 
\begin{equation*}
    (\mathcal{K} (s))_{i_1 i_2 i_3} =  \sum_{\tau \in \mathcal{T}(i_1, i_2, i_3)} 
    \frac{s^{\ell_{i_1}(\tau)+\ell_{i_2}(\tau)+\ell_{i_3}(\tau)+1}}{\ell_{i_1}(\tau)! \ell_{i_2}(\tau)! \ell_{i_3}(\tau)! (\ell_{i_1}(\tau)+\ell_{i_2}(\tau)+\ell_{i_3}(\tau)+1)}   
    \omega(M,\mathcal{C},\tau), 
\end{equation*}
as we wanted to prove. 
\end{proof}

\begin{proposition} 
    Let $G = ([d],E,B)$ be a $k$-th-order mixed graph and let $(M,\mathcal{C})$ be a pair of a $d\times d$ stable matrix and $\mathcal{C}$ $k$-th-order tensor that is compatible with $G$ then for $\mathcal{K}(s)$ as defined by equation~\eqref{eq::K(s)eq} we obtain 
\begin{align*}
(\mathcal{K} (s))_{i_1 \dots i_k} =  \sum_{\tau \in \mathcal{T}(i_1, i_2, \dots ,i_n)} 
\frac{s^{ \sum_{j = 1}^k \ell_{i_j}(\tau) +1}}{ \prod_{j = 1}^{k} \ell_{i_j}(\tau)! \left(\sum_{j = 1}^k \ell_{i_j}(\tau)+1 \right)}   
\omega(M,\mathcal{C},\tau).  
\end{align*}
\end{proposition}

\begin{remark}
    The primary usefulness of this result is that the $k$-th-order cumulant $(\mathcal{K} (s))_{i_1 \dots i_k} = 0$ if there is no $k$-trek between $i_1, \dots, i_k$ and by taking the limit this also applies to $\mathcal{K} $. 
\end{remark}

If we make additional assumptions on $M$ we will be allowed to interchange the sum and limit operation to obtain a trek formula for the cumulant directly as is done for $k=2$ by \cite{hansen2024trekrulelyapunovequation}. For any $M$ we can decompose it by subtracting a scaled version of the identity matrix, $I_{d \times d}$,
\begin{equation}
    \label{eq::Mdecomp}
    M = \Lambda - s \cdot I_{d \times d}. 
\end{equation}
One way to choose $s$ would be to choose it such that $s \cdot I_{d \times d}$ equals one or more of the diagonal elements of $M$. Then there is one less non-zero entry in $\Lambda$ than in $M$, which might  make the trek formulas simpler, because there will be fewer treks. Another choice would be $s = 1$.
Then as long as $M_{ii} \neq -1$, $\Lambda$ and $M$ would have the same zero-patterns and therefore the same treks. We will let $\omega(\Lambda, \mathcal{C}, \tau)$ denote the trek polynomial for $\Lambda$ independently of whether it will be the same as for $M$. In the case of $s = 1$, this yields the following trek rule. 
\begin{proposition}
\label{prop::LambdaTrekRuleThird}
    Let $G = ([d],E,B)$ be a third-order mixed graph and let $(M,\mathcal{C})$ be a pair of a $d\times d$ stable matrix and a $k$-th-order tensor that is compatible with $G$, 
    and let $\Lambda$ be defined by equation \eqref{eq::Mdecomp}. 
    If $\Lambda$ has spectral radius strictly less than $1$, then for $\mathcal{K}  \in \mathcal{M}^3_{G}$

    \begin{align*}
        (\mathcal{K} )_{i_1 i_2 i_3} = \sum_{\tau \in \mathcal{T}(i_1, i_2, i_3)} \omega(\Lambda, \mathcal{C}, \tau)  \frac{3^{-\ell_{i_1}(\tau)-\ell_{i_2}(\tau)-\ell_{i_3}(\tau)-1}  \Gamma(\ell_{i_1}(\tau)+\ell_{i_2}(\tau)+\ell_{i_3}(\tau)+1)}{\ell_{i_1}(\tau)! \ell_{i_2}(\tau)! \ell_{i_3}(\tau)!}. 
    \end{align*}
\end{proposition}   

\begin{proof}
    We obtain the following using the definition of $\Lambda$, the matrix exponential and the gamma distribution, and that we can interchange the sum and integral 
    because of the spectral radius assumption. 
    \begin{align*}
        \mathcal{K}  &= \int_{0}^{\infty} \mathcal{C} \times_1 e^{M t} \times_2 e^{M t} \times_{3} e^{M t} \mathrm{d} t \\
        &= \int_{0}^{\infty} \mathcal{C} \times_1 e^{(\Lambda - I_{d \times d}) t} \times_2 e^{(\Lambda -I_{d \times d}) t}  \times_{3} e^{(\Lambda - I_{d \times d}) t} \mathrm{d} t \\
        &= \int_{0}^{\infty} e^{-3 t } \cdot \mathcal{C} \times_1 e^{\Lambda t} \times_2 e^{\Lambda t}  \times_{3} e^{\Lambda t} \mathrm{d} t \\
        &= \int_{0}^{\infty} e^{-3 t } \cdot \mathcal{C} \times_1 \sum_{n = 0}^{\infty} \frac{t^n}{n!} \Lambda^n \times_2 
        \sum_{m = 0}^{\infty} \frac{t^m}{m!} \Lambda^m \times_3  \sum_{k = 0}^{\infty} \frac{t^k}{k!} \Lambda^k \mathrm{d} t \\ 
        &= \sum_{n = 0}^{\infty} \sum_{m = 0}^{\infty} \sum_{k = 0}^{\infty} \int_{0}^{\infty} \frac{t^{n+m+k} e^{-3t}}{n! m! k!}  
        \mathcal{C} \times_1  \Lambda^n \times_2 \Lambda^m \times_3  \Lambda^k \mathrm{d} t \\ 
        &=  \sum_{n = 0}^{\infty} \sum_{m = 0}^{\infty} \sum_{k = 0}^{\infty} \frac{3^{-n-m-k-1}  \Gamma(n+m+k+1)}{n! m! k!} 
        \mathcal{C} \times_1  \Lambda^n \times_2 \Lambda^m \times_3  \Lambda^k. 
    \end{align*}
We take the $(i_1 i_2 i_3)$-th entry and use the same derivation as in the proof of Proposition \ref{prop::ThirdOrderTrekwithS} to obtain an expression for  $\mathcal{C} \times_1  \Lambda^n \times_2 \Lambda^m \times_3  \Lambda^k$
by $3$-treks to obtain that 
\begin{align*}
    (\mathcal{K} )_{i_1 i_2 i_3} = \sum_{\tau \in \mathcal{T}(i_1, i_2, i_3)} \omega(\Lambda, \mathcal{C}, \tau)  \frac{3^{-\ell_{i_1}(\tau)-\ell_{i_2}(\tau)-\ell_{i_3}(\tau)-1}  \Gamma(\ell_{i_1}(\tau)+\ell_{i_2}(\tau)+\ell_{i_3}(\tau)+1)}{\ell_{i_1}(\tau)! \ell_{i_2}(\tau)! \ell_{i_3}(\tau)!},
\end{align*}
as we wanted to prove. 
\end{proof}

The above calculation can easily be carried out for a general decomposition of $M$ as given in equation \eqref{eq::Mdecomp} to obtain a trek rule that might be more useful in specific cases. An example of this is Lemma \ref{lem::simplifiedtrekruleDAG}. The trickiest part in terms of using this trek rule is to ensure that $\Lambda$ has spectral radius strictly less than 1, which is something that can be more easily ensured when $M$ is lower-triangular, i.e., the graph is a DAG with self-loops, as in Lemma \ref{lem::simplifiedtrekruleDAG}.
Below is the generalization for higher-order $k$. 
\begin{proposition}
\label{prop::LambdaTrekRuleGeneral}
  Let $G = ([d],E,B)$ be a $k$-th-order mixed graph and let $(M,C)$ be a pair of a $d\times d$ stable matrix and a $k$-th-order tensor that is compatible with $G$, 
    and let $\Lambda$ be defined by equation \eqref{eq::Mdecomp}. 
    If $\Lambda$ has spectral radius strictly less than $1$, then for $\mathcal{K}  \in \mathcal{M}^k_{G}$
    \begin{align*}
        \mathcal{K} _{i_1 \dots i_k} = \sum_{\tau \in \mathcal{T}(i_1, i_2, \dots, i_k)} \omega(\Lambda, \mathcal{C}, \tau)  \frac{k^{-\sum_{j = 1}^k \ell_{i_j}(\tau)-1}  \Gamma \left(\sum_{j = 1}^k \ell_{i_j}(\tau)+1 \right)}{ \prod_{j = 1}^{k}\ell_{i_j}(\tau)! }. 
    \end{align*}
\end{proposition}

\subsection{Additional identifiability results}
\label{Appendix::AddtionalIDResults}
In this section we provide an additional identifiability result for continuous Lyapunov models. 
Contrary to Theorem \ref{thm::GenericIdResult}, we assume that a given cumulant of order at least three of the Lévy process is known, as explored by \cite{Dettling:2023} in the covariance case. In this case we show generic identifiability of $M$ for any graph with all self-loops. 

\begin{proposition}
    \label{prop::ThirdOrderID}
    Let $G = ([d], E)$  be any directed graph with all self-loops and $r \geq 3$ an integer then the $r$-th-order continuous Lyapunov model $\mathcal{M}^r_{G}$ where $\mathcal{C}_r$ is assumed known and diagonal with non-zero diagonal is generically identifiable. 
\end{proposition}
\begin{proof}
    To show generic identifiability, we exhibit a choice of $M$ in $\mathbb{R}^{E}_{\mathrm{stab}}$ such that a suitable submatrix of $A_r(\mathcal{K} )$ as defined by equation \eqref{eq::VecThirdOrder} has rank $|E|$, so that we can solve the system in equation \eqref{eq::VecThirdOrder} for the non-zero entries of $M$.

    We let $M$ be any stable diagonal matrix (so it has negative values on the diagonal). This choice of $M$ will be in $\mathbb{R}^{E}_{\mathrm{stab}}$ for any directed graph with all self-loops, which we have assumed the graph to have. When $M$ is diagonal (with all diagonal elements non-zero) it follows by the general trek-rule that exactly $\mathcal{K} _{i \dots i} \neq 0$ for all $i$ but all other entries of $\mathcal{K} $ are zero since $\mathcal{C}_r$ is assumed diagonal.  

    Consider the submatrix $A$ of $A_r(\mathcal{K} )$ with the columns indexed by $i \rightarrow j \in E$ 
    and rows indexed by  $(i \dots i j)$ for $i \rightarrow j \in E$. This especially includes the diagonal rows $(i \dots i)$ for all $i \in [d]$). 

    By equation \eqref{eq::A3KEntries} we see that see that 
    $A_r(\mathcal{K} )_{(i \dots ij), (i \rightarrow j)} = \mathcal{K} _{i \dots i}$ for $i \neq j$ and $i \rightarrow j \in E$ and that $A_r(\mathcal{K} )_{(i \dots ij), (\alpha \rightarrow \beta)}$ is never equal to a diagonal entry of $\mathcal{K} $ if $(\alpha \rightarrow \beta) \neq (i \rightarrow j)$ by equation \eqref{eq::A3KEntries}. 
    
    Furthermore, $A_r(\mathcal{K} )_{(i \dots i), i \rightarrow i} = r \mathcal{K} _{i \dots i}$ and for any other column $A_r(\mathcal{K} )_{(i \dots i), (\alpha \rightarrow \beta)}$ is never equal to a diagonal entry of $\mathcal{K} $.
    
    Thus, we conclude that the submatrix $A$ has diagonal entries of $\mathcal{K} $ (upto multiplication by $r$) on its diagonal and all off-diagonal entries equal to off-diagonal entries of $\mathcal{K} $. Thus, by picking $M$ to be diagonal with negative diagonal the determinant of $A$ is just a product of diagonal entries of $\mathcal{K} $
    \begin{equation*}
        \mathrm{det}(A) = \prod_{i \rightarrow j \in E, i \neq j} \mathcal{K} _{i \dots i} \prod_{i \in [d]} r \mathcal{K} _{i \dots i},
    \end{equation*}
    which are all non-zero by choice of $M$ and the diagonal of $\mathcal{C}_r$ being non-zero . Therefore, $\mathrm{det}(A) \neq 0$ for this choice of $M$, which completes the proof.   
\end{proof}

\subsection{Auxiliary results when the Lévy process is a compound Poisson process} \label{sec:AuxResultsCompoundPoisson}

Since $Z_1$ is infinitely divisible, its characteristic function is of the form
$\mathbb{E}(e^{iz^TZ_1}) = e^{\Psi(z)}$ where $\Psi$ is the Lévy 
exponent. Its $k$-th-order cumulants can then be expressed as
\begin{equation}
    \label{eq:cumulants_exponent}
	\mathcal{C}_k = (-i)^n D^k\Psi(0)
\end{equation}
provided that $\Psi$ is $k$ times differentiable in $0$. If the Lévy process is 
a compound Poisson process, its Lévy exponent is of the form 
\[
    \Psi(z) = \boldsymbol{\lambda} \int e^{iz^Tu} P(\mathrm{d} u), 
\]
where $\boldsymbol{\lambda}  > 0$ is the rate parameter of the homogeneous Poisson process, and 
the probability measure $P$ is the jump distribution. Using \eqref{eq:cumulants_exponent}
we find that for a compound Poisson process, the $k$-th-order cumulant for $k \geq 2$ equals

\begin{equation}
    \label{eq:comp_pois_cum}
	\mathcal{C}_k = \boldsymbol{\lambda}  \int \underbrace{u \otimes u \otimes \ldots \otimes u}_{k \text{ factors}} P(\mathrm{d} u). 
\end{equation}

Suppose next that the coordinates of $Z$ are, in fact, independent univariate compound Poisson processes.
The $i$-th coordinate, denoted $Z^{(i)}$, has Lévy measure $\lambda_i P_i$, where $\lambda_i > 0$ 
and $P_i$ is the univariate probability distribution of the jumps of $Z^{(i)}$. Then $Z$ is a 
compound Poisson process with $\boldsymbol{\lambda}  = \lambda_1 + \ldots + \lambda_p$
and
$$
	P = \sum_{i=1}^d \frac{\lambda_i}{\boldsymbol{\lambda} } \tilde{P}_i,
$$
where the probability measure $\tilde{P}_i$ on $\mathbb{R}^d$ is given by
$$
	\int f(u) \tilde{P}_i(\mathrm{d}u) = \int f(0, \ldots, 0, u_i, 0, \ldots, 0) P_i(\mathrm{d} u_i).
$$
From this it follows that $\mathcal{C}_k$ is diagonal with
\begin{equation}
    \label{eq:comp_pois_diag_cum}
	\mathrm{diag}(\mathcal{C}_k) = \left(
	\lambda_1 \int u^k \ P_1(\mathrm{d} u), \ldots, \lambda_p \int u^k \ P_p(\mathrm{d} u)\right).
\end{equation}
Note that for a compound Poisson process with independent coordinates, it is with probability one
always only one coordinate that jumps at a time.

We can use specific compound Poisson processes to prove the following about the surjectivity of $\pi_{2,r}$. 
\begin{proposition} \label{prop::piSurjective}
Let r be an integer such that $r \geq 3$. 
\begin{enumerate}[label = (\roman*),leftmargin=*, widest=iii]
        \item If $r$ is odd the map $\pi_{2,r}$ is surjective. 
        \item if $r$ is even the map $\pi_{2,r}$ is not surjective but does map onto a fully dimensional subset of $\Theta_G$.
    \end{enumerate}
\end{proposition}
\begin{proof}
    The proof is constructive and only focuses on the part of the map $Q_1 \mapsto (\mathcal{C}_2, \mathcal{C}_r)$, since $\pi_{2,r}$ is the identity on the $M$ coordinate. 

    Furthermore, since $\mathcal{C}_2$ and $\mathcal{C}_r)$ are assumed diagonal, we can construct the example by considering a Lévy process, and corresponding infinitely divisible distribution, with independent coordinates. For this reason, it is enough to construct a univariate 
    Lévy process with a given covariance $c_2$ and $r$-th-order cumulant $c_r$. 

     We let the Lévy process be a compound Poisson process. 
     By equation \eqref{eq:comp_pois_diag_cum} the cumulants of the compound Poisson can be found as the rate parameter multiplied by the raw moments of the jump distribution, $P$, 
    \begin{equation*}
        \mathcal{C}_r = \lambda \int u^k \ P(\mathrm{d} u).
    \end{equation*}
    Thus, (i) and (ii) follow by (i) and (ii) in Lemma \ref{lem::rawmomentsexample}. 
\end{proof}

\begin{lemma} \label{lem::rawmomentsexample}
    Let $r$ be an integer such that $r \geq 3$ and $(c_2, c_r) \in \mathbb{R}^2$ with $c_2 > 0$ and $c_r \neq 0$. 
    \begin{enumerate}[label = (\roman*),leftmargin=*, widest=iii]
        \item If $r$ is odd there exists a probability distribution $P$ such that $\mathbb{E}[X^2] = c_2$ and $\mathbb{E}[X^r] = c_r$  for $X \sim P$. 
        \item if $r$ is even and $|c_r| \geq c_2^{r/2}$ there exists a probability distribution $P$ such that $\mathbb{E}[X^2] = c_2$ and $\mathbb{E}[X^r] = c_r$  for $X \sim P$.
    \end{enumerate}
\end{lemma}
\begin{proof}
    The proof is constructive. 
    We provide an example of 2-point distribution $P$ such that $\mathbb{E}[X^2] = c_2$ and $\mathbb{E}[X^r] = c_r$ if $X \sim P$, with the extra condition $|c_r| \geq c_2^{r/2}$ if $r$ is even. Let
    \begin{align*}
        P( X = a) = p, \qquad P(X = b) = 1-p
    \end{align*}
    for $a,b \in \mathbb{R}$ and $p \in [0,1]$. Then the condition $\mathbb{E}[X^2] = c_2$ and $\mathbb{E}[X^r] = c_r$ is equivalent to being able to solve the following system of two polynomial equations
    \begin{align} \label{eq::momenteq}
        c_2 = a^2 p + b^2 (1-p), \qquad c_r = a^r p + b^r (1-p)
    \end{align}
    for $a$, $b$ and $p$. 

    If $|c_r| \geq c_2^{r/2}$ then 
    \begin{align*}
        b = 0, \quad a = \mathrm{sign}(c_r) \cdot \left| \frac{c_r}{c_2} \right|^{1/(r-2)}, \quad p = c_2 \left| \frac{c_2}{c_r} \right|^{2/(r-2)}
    \end{align*}
    is a solution to the polynomial equations \eqref{eq::momenteq}, and the assumption $|c_r| \geq c_2^{r/2}$ guarantees that $p \in (0,1]$. If $|c_r| = c_2^{r/2}$ we have $p = 1$, and the distribution becomes a 1-point distribution. This proves (ii) and acts as the first case of (i). 

    If $|c_r| < c_2^{r/2}$ and $r$ is odd then
    \begin{align*}
        b = - \sqrt{c_2}, \quad a = \sqrt{c_2}, \quad p = \frac{1}{2} \left( 1+ \frac{c_r}{c_2^{r/2}} \right)
    \end{align*}
    is a solution to the polynomial equations \eqref{eq::momenteq}. The assumption that $|c_r| < c_2^{r/2}$ guarantees that $p \in (0,1)$. If $r$ were even, $a^r = b^r$, and it follows that $a$, $b$ and $p$ is only a solution for odd $r$. 
    This concludes the proof of (i). 
\end{proof}

\begin{remark}
The inequality for even $r$ in (ii) is necessary in the sense that it can be seen by Jensen's inequality that $\mathbb{E}[X^{r}] \geq \mathbb{E}[X^2]^{r/2}$ for any probability distribution when $r$ is even. 
\end{remark}

\subsection{Details regarding the numerical experiments and additional results}
\label{sec:numerical_details} 
Simulations from the steady-state distribution are obtained by using the representation 
\eqref{eq:MSD}. Supposing that $M$ can be diagonalized as $M = Q D Q^{-1}$, where
$D = \mathrm{diag}(\delta_1, \ldots, \delta_d)$ is a diagonal matrix of
complex eigenvalues, we get that
\begin{equation}
    \label{eq:MSD_diag}
	X = \int_0^{\infty}  e^{sM} \mathrm{d} Z_s                
	  = \int_0^{\infty}  Q e^{sD} Q^{-1} \mathrm{d} Z_s 
	  = Q \begin{pmatrix}
	      \int_0^{\infty} e^{\delta_1 s} \mathrm{d} \tilde{Z}_{s}^{(1)} \\
			\vdots                                                          \\
			\int_0^{\infty} e^{\delta_d s} \mathrm{d} \tilde{Z}_{s}^{(d)}
	  \end{pmatrix}
\end{equation}
where $\tilde{Z}_{s}^{(i)} = (Q^{-1} Z_s)_i$. 

Our numerical experiments use simulations from $M$-selfdecomposable distributions given in terms of:
\begin{enumerate}
	\item A Lévy process $Z$ with independent coordinates, each being a compound Poisson process with beta distributed jumps.
	\item A stable $M$-matrix parametrized by a correlation parameter $\rho$ and a parameter $\gamma$ controlling the asymmetry of $M$.
\end{enumerate}
The $i$-th compound Poisson process is
$$
	Z^{(i)}_t = \sum_{j=1}^{N^{(i)}_t} J_j^{(i)},
$$
where $N^{(i)}$ is a homogeneous Poisson process with rate $\lambda_i > 0$ and
$J_1^{(i)}, J_2^{(i)}, \ldots $ are i.i.d. beta distributed jumps. In this case, 
each univariate stochastic integral in \eqref{eq:MSD_diag} is an infinite sum over jumps of $\tilde{Z}_{s}^{(i)}$, 
\begin{equation}
\label{eq:jump_rep}
\int_0^{\infty} e^{\delta_i s} \mathrm{d} \tilde{Z}_{s}^{(i)} = \sum_{j=1}^{\infty}  
\tilde{J}_{j}^{(i)} e^{\delta_i s_j}, 
\end{equation}
with the exponential factor decaying rapidly due to stability of $M$. Simulations of $X$ 
are implemented via truncation of the infinite series \eqref{eq:jump_rep}.

The $i$-th Poisson process has rate $\lambda_i > 0$, and all $d$ independent 
compound Poisson processes have beta distributed jumps sharing the parameters
$\nu > 0$ and $\mu \in (0, 1)$, known as the mean and size parameters, respectively, 
of a beta distribution. The shape parameters 
are given as $(\alpha, \beta) = (\mu \nu, (1 - \mu) \nu)$ in terms of these parameters. 

Since the $k$-th raw moment of a beta distribution is 
\[
    \int u^k P(\mathrm{d} u) = \prod_{r=0}^{k-1} \frac{\mu \nu + r}{v + r},
\]
the diagonal $k$-th-order cumulants are by \eqref{eq:comp_pois_diag_cum} given as 
$\mathcal{C}_{1,i} = \mathbb{E}(Z^{(i)}_1) = \lambda \mu$ and 
\[
    \mathcal{C}_{k, i\ldots i} =  \lambda \prod_{r=0}^{k-1} \frac{\mu \nu + r}{\nu + r} 
    = \frac{\mu \nu + k - 1}{\nu + k - 1} \mathcal{C}_{k - 1, i\ldots i}
\]
for $k \geq 2$ and $i = 1, \ldots, d$. The $d \times d$ matrix $M$ is parametrized as 
$$
	M =  \left(\gamma E^{\mathrm{skew}} -  d I \right)
	\left(I - \eta E \right) =
     \left(
	\begin{array}{cccc}
			-d & \gamma                           & \ldots & \gamma                           \\
			- \gamma                         & -d & \ldots & \gamma                           \\
			\vdots                           & \vdots                           & \ddots & \vdots                           \\
			- \gamma                         & - \gamma                         & \ldots & -d
		\end{array} \right)
	\left(\begin{array}{cccc}
			1 - \eta & - \eta   & \ldots & - \eta   \\
			- \eta   & 1 - \eta & \ldots & - \eta   \\
			\vdots   & \vdots   & \ddots & \vdots   \\
			- \eta   & - \eta   & \ldots & 1 - \eta
		\end{array} \right).
$$
where $I$ is the identity matrix, $E$ is the matrix with $E_{ij} = 1$, and $E^{\mathrm{skew}}$
is the skew symmetric matrix with $E^{\mathrm{skew}}_{ij} = 1$ for $i < j$. The parameters are 
$\gamma \in \mathbb{R}$ and $\eta < 1/d$.
\begin{figure}
    \centering
    \includegraphics[width=\linewidth]{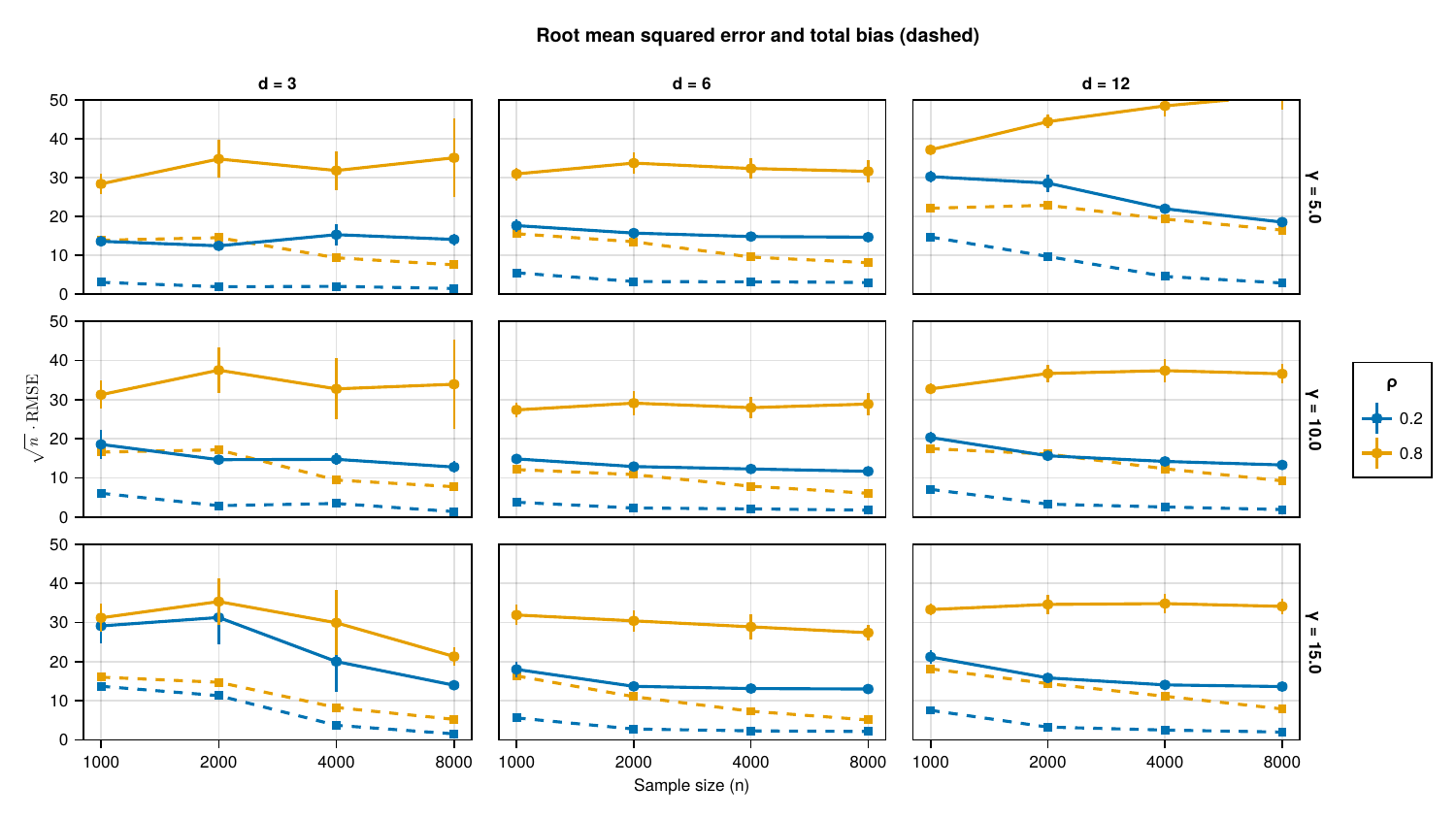}
    \caption{Estimation errors for the least singular value estimator based on all second-, third- and fourth-order cumulants. Circles connected with full lines show the $\sqrt{n}$-scaled root mean squared error, while the squares connected with dashed lines show the $\sqrt{n}$-scaled total bias.}
    \label{fig:estimation_errors_k4}
\end{figure}
If $\lambda_1 = \ldots = \lambda_d = \lambda$, then with 
$$
    c = \frac{\lambda \mu (\mu \nu + 1)}{2d (\nu + 1)}  
$$
we have $\mathcal{C}_{2, ii} = 2 d c$, and it is straightforward to see that the 
positive definite matrix 
$\Sigma = c \left(I - \eta E \right)^{-1}$ solves the second-nd order Lyapunov equation 
for the choice of $M$ above. By the Sherman-Morrison formula
$$
	\Sigma = c \left(I - \eta E \right)^{-1} = c \left(I + \frac{\eta}{1 - d \eta} E\right).
$$
The variance of all coordinates is
$$
	\sigma^2 = c\left(1 +  \frac{\eta}{1 - d \eta}\right) = c \frac{1 + (1 - d)\eta}{1 - d \eta},
$$
and the correlation between any two coordinates is
$$
	\rho = \frac{\eta}{1 + (1 - d)\eta} = \frac{1}{\eta^{-1} + (1 - d)}.
$$
The correlation is thus entirely determined by $\eta$, it increases monotonically with $\eta$,
it is positive for $\eta \in (0, 1/d)$, tending to $1$ for $\eta \to 1/d$, it is $0$ for $\eta = 0$,
and it is negative for $\eta < 0$, tending to $- 1 / (d - 1)$ for $\eta \to -\infty$. The parameter
$\gamma$ does not affect the covariance matrix but determines the asymmetry of $M$. 
For the numerical experiments, $M$ is parametrized by the correlation parameter 
$\rho$ via 
$$
    \eta = \frac{\rho}{1 + \rho (d - 1)}
$$
for $\rho \in (-1/(d-1), 1)$.

\newpage
\bibliographystyle{biometrika}
\bibliography{paper-ref}

\end{document}